\documentclass[11pt]{article}
\usepackage[utf8]{inputenc}
\usepackage{lmodern}
\usepackage[normalem]{ulem} 
\usepackage{float}
\usepackage[margin=0.8in]{geometry}
\usepackage{enumitem}

\usepackage[utf8]{inputenc} 

\usepackage{tocloft}

\usepackage{graphicx, titlesec}
\usepackage{amscd,amsmath,amsthm,amssymb}
\usepackage{verbatim}
\usepackage[all]{xy}
\usepackage{booktabs}
\usepackage{tikz-cd}
\usepackage{ytableau}
\usepackage[algonl,boxed,norelsize,lined]{algorithm2e}
\usepackage[bookmarksnumbered=true]{hyperref} 
\hypersetup{
     colorlinks = true,
     linkcolor = blue,
     anchorcolor = blue,
     citecolor = teal,
     filecolor = blue,
     urlcolor = blue
     }
\DeclareMathOperator{\Pos}{Pos}
\DeclareMathOperator{\MPos}{MPos}
\DeclareMathOperator{\Gr}{Gr}
\DeclareMathOperator{\Mat}{Mat}
\newcommand{\mI}{\mathcal{I}}

\newcommand{\mD}{\mathcal{D}}
\newcommand{\mM}{\mathcal{M}}

\newcommand{\mB}{\mathcal{B}}

\newcommand{\mF}{\mathcal{F}}

\definecolor{ao}{rgb}{0.0, 0.5, 0.0}

\newcommand\xqed[1]{%
  \leavevmode\unskip\penalty9999 \hbox{}\nobreak\hfill
  \quad\hbox{#1}}
\newcommand\demo{\xqed{$\triangle$}}

\newtheorem{theorem}{Theorem}[section]

\theoremstyle{definition}
\newtheorem{definition}[theorem]{Definition}
\newtheorem{notation}[theorem]{Notation}
\newtheorem{example}[theorem]{Example}
\newtheorem{remark}[theorem]{Remark}
\newtheorem{question}[theorem]{Question}

\theoremstyle{plain}
\newtheorem{lemma}[theorem]{Lemma}
\newtheorem{proposition}[theorem]{Proposition}
\newtheorem{corollary}[theorem]{Corollary}

\title{Computing positroid cells in the Grassmannian of lines,
\\ their boundaries and their intersections}
\author{Susama Agarwala, Fatemeh Mohammadi, and Francesca Zaffalon
}

\begin{document}
\maketitle 
\begin{abstract}
Positroids are families of matroids introduced by Postnikov in the study of nonnegative Grassmannians. In particular, positroids enumerate a CW decomposition of the totally nonnegative Grassmannian. Furthermore, Postnikov has identified several families of combinatorial objects in bijections with positroids. We will provide yet another characterization of positroids for Gr$_{\geq 0}(2,n)$, the Grassmannians of lines, in terms of certain graphs. We use this characterization to compute the dimension and the boundary of positroid cells. 
This also leads to a combinatorial description of 
the intersection of positroid cells, that is easily computable. Our techniques rely on determining different ways to enlarge a given collection of subsets of $\{1,\ldots,n\}$ to represent the dependent sets of a positroid, that is the dependencies among the columns of a matrix with nonnegative maximal minors. Furthermore, we provide an algorithm to compute all the maximal positroids contained in a set.
\end{abstract}
{\hypersetup{linkcolor=black}
{\tableofcontents}}

\section{Introduction}

In this note, we take an algorithmic approach to study positroids, a class of matroids that can be represented by the columns of a real matrix with nonnegative maximal minors.
More explicitly, for the Grassmannians of lines $\Gr(2,n)$, we provide a characterization for  families of $2$-subsets of $\{1,\dots,n\}$ to represent the dependencies of a positroid and study the positroids whose dependent sets are compatible with them.
We then provide an algorithm to identify all maximal positroids among them. 
This, in turn, enables us to compute the dimension of positroid cells, their boundaries, and their intersections.

\medskip

Positroid varieties in $\Gr(k,n)$ and positroid cells in $\Gr_{\geq 0}(k,n)$, along with their combinatorial and geometric properties, have been the subject of active research over the past twenty years. Building on the seminal work on {\em positivity} by Lusztig \cite{lusztig} and by Fomin and Zelevinsky \cite{fomin1999double}, Postnikov \cite{postnikov2006total} introduced positroids in the study of the cell decompositions of the nonnegative Grassmannian and showed that they possess remarkable algebraic, geometric, and combinatorial structures. Geometrically, the restriction of the matroid stratification of $\Gr(k,n)$ given in \cite{gelfand1987combinatorial} to the nonnegative Grassmannian provides a decomposition of $\Gr_{\geq 0}(k,n)$ into so-called \emph{positroid cells} \cite{postnikov2006total,PSW,rietsch}. Each positroid cell is a topological cell homeomorphic to a ball \cite{galashin2021totally}, and moreover, the positroid cells glue together to form a CW complex \cite{PSW}. Since Postnikov's work \cite{postnikov2006total}, positroids have been extensively studied in mathematics and have also appeared in various other fields, including amplituhedron theory (see, e.g., \cite{Arkani-Hamed:2016byb} and references therein).

More algebraically, positroid varieties were initially identified by Knutson, Lam, and Speyer \cite{knutson2013positroid} as a refinement of the Bruhat and Richardson variety decompositions of the Grassmannians. When restricted to the nonnegative Grassmannian $\Gr_{\geq 0}(k,n)$ (the points in the Grassmannian where all the Plücker coordinates are nonnegative), these positroid varieties form a CW complex \cite{galashin2021totally}, where the intersection of the closures of two positroid cells is itself a closed positroid cell. However, little is known about the poset structure of positroid cells induced by inclusion of closures. Some insight into this structure may be gained from the Bruhat order, which underlies \cite{knutson2013positroid, talaska2013network, marsh2004parametrizations, Towercommunication}.

\medskip
Postnikov \cite{postnikov2006total} identified a family of combinatorial tools for enumerating and analyzing the relationships between positroid cells that are particularly useful for understanding the connections among the geometric subspaces in $\Gr_{\geq 0}(k,n)$, the Plücker coordinates that define them, and the Bruhat intervals that define the positroid varieties. This work has generated a large body of literature directly establishing equivalences between the various enumeration methodologies introduced therein \cite{oh2011positroids, agarwala2020algorithm, talaska2013network, marcott2020combinatorics, casteels2017grassmann, mohammadi2024combinatorics}, among others. These developments have made it easier to work with positroid cells, since it is generally easier to translate a set of Plücker coordinates into one of these combinatorial objects and thereby extract the desired geometric or combinatorial information. However, to date, there is no good algorithm to determine when one of these combinatorial objects corresponds to a positroid lying on the boundary of the closure of another.

\medskip

This work overcomes this challenge for $\Gr_{\geq 0}(2,n)$ by answering the 
following questions.

\begin{question}\label{question:matroid}
Given a collection $\mathcal{D}$ of $2$-subsets of $[n]:=\{1,\ldots,n\}$, does $\mathcal{D}$ represent the dependent sets (i.e., vanishing Plücker coordinates) of a matroid? Equivalently, is there a $2\times n$ matrix whose columns $i$ and $j$ are linearly dependent if and only if $\{i,j\}\in \mathcal{D}$? Furthermore, does such a matrix exist with nonnegative $2$-minors? 
\end{question}

If $\mathcal{D}$ does not represent the dependent sets of a matroid or a positroid, we ask:
\begin{question}\label{question:max}
Given a collection $\mathcal{D}$ of $2$-subsets of $[n]$, in how many ways can we enlarge $\mathcal{D}$ so that it represents the dependent sets of a positroid? Equivalently, which positroids have collections of dependent sets containing $\mathcal{D}$? Among these, which are maximal with respect to inclusion?
\end{question}

We answer the above questions by taking an algorithmic approach. We first represent any collection of dependent sets as the edge set of a simple graph on $n$ vertices. We then use this graphical representation to characterize graphs encoding the dependent sets of a matroid or a positroid (see Lemma~\ref{lemma:completecomp} and Proposition~\ref{prop:niceness}). In Proposition~\ref{prop:matT}, we provide a procedure for finding matroids whose dependent sets contain $\mD$, and we use these matroids to identify all maximal positroids with this property. In particular, Algorithm~\ref{alg:posT} enumerates all such positroids, thereby answering Question~\ref{question:max}. This allows us to determine all codimension-one boundaries of a given positroid cell, and in particular, the maximal positroids contained in the intersection of positroid cells (see Propositions~\ref{res:1boundaries} and \ref{prop:intersection}). Moreover, in Proposition~\ref{cor:dimension} we compute the dimension of each positroid cell in terms of graph invariants.

We finish the introduction with an outline of the paper. Section~\ref{sec:pre} fixes notation for the nonnegative Grassmannians and positroids and outlines an explicit bijection between positroids and Le diagrams. In Section~\ref{sec:results}, we present our main results.  
More precisely, \S\ref{sec:cells} contains our graphical representation of sets, and \S\ref{sec:niceanddim} the characterization of positroids in terms of these graphs. In \S\ref{sec:containment}, 
we present a procedure to identify the collection of maximal positroids that are compatible with a given set of dependencies. In \S\ref{sec:dim}, we give a formula to compute the dimension of positroid cells. In \S\ref{sec:boundaries}, we apply our main results to compute the codimension-one boundaries of positroid cells, and in \S\ref{sec:intersections}, we compute the maximal positroid cells contained in the intersections of the closures of given positroid cells.

\medskip
\noindent{\bf Acknowledgments.} The authors would like to thank Matteo Parisi and Leonid Monin for helpful discussions on \cite{mohammadi2021triangulations} and Question 1.2 above. F.M. and F.Z. were partially supported by UGent BOF grant STA/201909/038, FWO grants (G023721N, G0F5921N), and iBOF/23/064 from KU Leuven.

\section{Preliminaries}\label{sec:pre}
In this section, we establish some general background on different means of defining positroid cells, many of which are equivalent, but are useful identifications in different contexts. For more details, we refer the reader to \cite{postnikov2006total} and \cite{oxley2006matroid}. 

\smallskip

Throughout we let $[n]=\{1,\ldots,n\}$ be cyclically ordered and we denote $\binom{[n]}{k}$ for the collection of $k$-subsets of $[n]$. Given a subset $\mD \subset \binom{[n]}{k}$, we denote $\mD^c$ for its complement in $\binom{[n]}{k}$, that is $\mD^c=\binom{[n]}{k} \backslash \mD$.

\subsection{Totally nonnegative Grassmannians}
The real Grassmannian $\Gr(k,n)$ is the space of all $k$-dimensional linear subspaces of $\mathbb{R}^n$. 
A point $V \in \Gr(k,n)$ can be represented by a full-rank $k \times n$ matrix with entries in $\mathbb{R}$. 
Let $X = (x_{ij})$ be such a matrix. 
For any $k$-subset $I = \{i_1,\ldots,i_k\}$, let $X_I$ denote the $k \times k$ submatrix of
$X$ with the column indices $i_1,\ldots,i_k$. 
The \emph{Pl\"ucker coordinates} of $V$ are $p_I(V) = \text{det}(X_I)$ for $I \in\binom{[n]}{k}$. Note that the Pl\"ucker coordinates do not depend on the choice of matrix $X$ representing the point $V$ in $\Gr(k,n)$ (up to simultaneous rescaling by a nonzero constant). 

\smallskip

An element in the real Grassmannian $\Gr(k,n)$ is called \emph{totally nonnegative} if it has a matrix representation whose (cyclically ordered) maximal minors are all
nonnegative. The subset defined by these elements is called the nonnegative Grassmannian, denoted by $\Gr_{\geq 0}(k,n)$.

\subsection{Matroids and positroids \label{sec:matroids}}

There are many equivalent ways to define matroids. However, for our discussion, it will be most useful to define them in terms of their bases.

\medskip

A matroid is a pair $\mathcal{M}=(E,\mathcal{B})$ where $E$ is a finite set called ground set and $\mathcal{B}$, called bases, is a nonempty collection of subsets of $E$ satisfying the exchange axiom for every pair of bases 
$B_1, B_2$ in $\mathcal{B}$:
\begin{itemize}
    \item $\text{If } b_1\in B_1\backslash B_2 \text{, then there exists } b_2\in B_2\backslash B_1 \text{ such that } B_1\backslash\{b_1\} \cup \{b_2\}\in \mathcal{B}.$
\end{itemize} 
As a consequence of the exchange axiom, every element of the bases $\mB$ has the same cardinality, which is called \textit{rank} of the matroid $\mathcal{M}$.
We can now define dependent and independent sets of matroids, as central concepts of this paper.
A set is called {\em independent} if and only if it is a subset of a basis, otherwise it is called {\em dependent}. In particular minimal dependent sets of $\mM$ are called {\em circuits} denoted by $\mathcal{C}(\mM)$.

\medskip

A matroid $\mathcal{M}=(E,\mathcal{B})$ is called {\em representable} over a field $\mathbb{K}$ if there exists a $k\times |E|$ matrix $A$ with entries in $\mathbb{K}$ such that, if we denote by $A[I]$ the submatrix obtained by selecting the columns indexed by the set $I\in \binom{|E|}{k}$, then
\[\det A[I] \neq 0 \text{ if and only if } I \in \mathcal{B}.\]
If there exists such a matrix $A$ which is a representative of an element in the nonnegative Grassmannian $\Gr_{\geq 0}(k,n)$, that is all its maximal minors are nonnegative, then $\mathcal{M}$ is called a \emph{positroid}. %

\begin{definition}[Positroid cells] \label{def:positroid}
For each $\mathcal{I}\subseteq \binom{[n]}{k}$, we define the set $S_+(\mathcal{I})$ as 
\[ S_+(\mathcal{I}) = \{GL_+(k) \cdot A \mid \det (A[I])>0 \text{ if } I \in \mathcal{I}, \text{ and } \det(A[I])=0 \text{ if } I \not\in \mI\}, \]
where $GL_+(k)$ is the set of invertible totally positive $k\times k$ matrices. If $S_+(\mathcal{I})\neq\emptyset$, then we call $\mathcal{I}$ a \emph{positroid} and $S_+(\mathcal{I})$ a \emph{positroid cell}.
\end{definition}

\begin{remark}\label{rmk:cyclicsym}
We recall from \cite[Lemma~3.3]{ardila2016positroids} that whether a matroid is a positroid depends essentially on the total order of its ground set; however, it is invariant under cyclic shifts of the ground set. 
Consider the column vectors $v_1,\ldots,v_n \in \mathbb{R}^k$ forming the $k \times n$ matrix $A = (v_1, \dots, v_n)$, and let $A' = (v_2,\dots, v_n, (-1)^{k-1} v_1)$. Note that $p_I(A) = p_{I'}(A')$ for every $k$-subset $I$ and its cyclic shift $I'$. This defines an action of the cyclic group $\mathbb{Z}/n\mathbb{Z}$ on $\Gr_{\geq 0}(k,n)$. In particular, for each $i$, the cyclic shift of the order $1 < \cdots < n$ to $i < i+1 < \cdots < n < 1 < \cdots < i-1$ preserves the class of positroids in $\Gr_{\geq 0}(k,n)$. 
\end{remark}

\begin{remark}[Dual matroids]
We recall that given a matroid $\mathcal{M} =(E,\mB)$, its dual is the matroid $\mathcal{M}^*=(E,\mB^*)$ with bases $\mB^* =\{E\backslash B \mid B \in \mB\}$. In particular, ${\rm rank}(\mathcal{M}^*)=n-{\rm rank}(M)$. Moreover, by \cite[Proposition 3.5]{ardila2016positroids} the dual of any positroid on the cyclically ordered set $E$ is also a positroid.
\end{remark}

In the following, we will focus on matroids and positroids of rank $2$. However, by the above remark, the results can be applied for matroids and positroids of rank $n-2$ by taking the dual matroids.

\subsection{A bijection between positroids and Le diagrams}\label{sec:Le}

There are several ways to index positroids, each with its own advantages and disadvantages. In particular, \cite{postnikov2006total} and \cite[\S4]{ardila2016positroids} describe various bijections among decorated permutations, Le diagrams, Grassmann necklaces, and equivalence classes of reduced plabic graphs. Here, we focus on giving a direct way to move between positroids, Le diagrams, and positive reduced subexpressions: such maps can be constructed by combining the previously mentioned bijections.


\begin{definition}[Le diagrams]
A \textit{Le diagram} is a Young diagram in which each box contains either a $+$ or a $0$, subject to the following condition: if a box contains a $0$, then at least one of the following holds:
\begin{itemize}
    \item Every box to its left in the same row contains a $0$;
    \item Every box above it in the same column contains a $0$.
\end{itemize}
\end{definition}

The positroids in $\Gr_{\geq 0}(k,n)$  
are in bijection with the Le diagrams that fit inside a $k \times (n-k)$ rectangle. 
The advantage of Le diagrams is that one can easily read off the dimension of a positroid cell by simply counting the number of $+$s. Furthermore, we can 
identify the bases 
of a positroid from its Le diagram. 

\medskip

In the following, we provide an algorithm to label the positroids with Le diagrams.

\medskip
\noindent{\bf Constructing a positroid from a Le diagram.}\label{page:algorithm}
To construct the positroid associated to a given Le diagram we can combine the correspondence between Le diagrams and planar directed networks shown in \cite[\S6]{postnikov2006total} with the bijection between planar directed networks and positroids in \cite[\S4]{postnikov2006total}. 

\smallskip

Let $L$ be a Le diagram fitting inside a $k \times (n-k)$ rectangle. Label each step along the southeast border of $L$ with $1,2,\dots,n$ starting from the north-east corner. Note that if $L$ has fewer than $k$ rows or fewer than $n-k$ columns, some of these steps will lie on the bounding $k \times (n-k)$ rectangle. Let $S$ and $T$ be the sets of labels of rows and columns, respectively. We associate a directed graph $\Gamma(L)$ to $L$ as follows, and we show that every such graph has a canonical positroid. To construct the graph $\Gamma(L)$:
\begin{enumerate}
    \item[(1)] Place a vertex next to each row and column label. Add a vertex in each square 
    containing a $+$;
    \item[(2)] Join any two consecutive vertices in the same row with a leftward pointing arrow. Join any two consecutive vertices in the same column with an arrow directed downwards.
\end{enumerate}
A \textit{path} in $\Gamma(L)$ is any path from a vertex $s \in S$ to a vertex $t \in T$ along the directed arrows. Two paths are called \textit{vertex-disjoint} if they do not have any vertex in common. If $I=\{i_1,\dots,i_r\}\subseteq S$ and $J=\{j_1,\dots,j_r\}\subseteq T$, then a \textit{vertex-disjoint path system} for $(I,J)$ is a collection of paths $i_1\to j_1,\dots , i_r \to j_r$ which are pairwise disjoint. Every such graph $\Gamma(L)$ defines a positroid $\mathcal{P}_L=([n],\mathcal{B})$ where
\[ B \in \mathcal{B} \iff |B|=k \text{ and } (S\backslash B, T\cap B) \text{ admits a vertex-disjoint path system.} \]

Given a Le diagram, we say that a box is in position $(i,j)$ if, assigning the labels to rows and columns as described above, the box lies in the row indexed $i$ and the column indexed $j$. For instance in the Le diagram in Example~\ref{ex:pos}, the box containing zero is in position $(1,5)$.

\medskip
\noindent{\bf Constructing a Le diagram from a positroid.} Let $\mB \subseteq \binom{[n]}{k}$ be the bases set of a positroid. To construct the Le diagram associated to $\mB$ we may combine the following two algorithms: first we construct the so-called Grassmann necklace $\mI_\mB=\{I_1,\ldots,I_n\}$ corresponding to the positroid $\mB$ as shown in \cite{postnikov2006total, Mcalmon2019FlatsOA}. Then we apply \cite[Algorithm~2]{agarwala2020algorithm} to obtain the corresponding Le diagram. 
\begin{itemize}
    \item[(1)] For $i \in [n]$, let $I_i$ be 
    the lexicographically minimal element of $\mB$ with respect to the shifted linear order $<_i$ on $[n]$, where 
    $i <_i i + 1 <_i  \cdots <_i n <_i 1 <_i\cdots
<_i i - 1$. 
    \item[(2)] Draw the Young diagram fitting inside a $k \times (n-k)$ rectangle such that labeling the southeast border as seen before, the row label is $I_1$;
    \item[(3)] For every $i$, $2 \leq i \leq n$, write:
    \[ I_1\backslash I_i =\{a_1 > a_2 >\dots >a_r\}, \quad I_i\backslash I_1 =\{b_1 < b_2 <\dots < b_r\}. \]
    Place a $``+"$ in each box $(a_1,b_1), \dots, (a_r,b_r)$.
    \item[(4)] After performing step (3) for all $i$, place a $``0"$ in every box remained unfilled.
\end{itemize}

\begin{example}\label{ex:pos}
Consider the positroid with bases set $\mathcal{B}\subset \binom{[6]}{2}$ given by
\[\{ \{1,4\}, \{1,5\}, \{1,6\}, \{2,4\}, \{2,5\}, \{2,6\}, \{3,4\}, \{3,5\}, \{3,6\}, \{4,6\}, \{5,6\}\}.\]
The associated Grassmann necklace is $\mI_\mB =(I_1,\dots, I_6)$ with
\[ I_1=\{1,4\}, \; I_2=\{2,4\}, \; I_3=\{3,4\}, \; I_4=\{4,6\}, \; I_5=\{5,6\}, \; I_6=\{1,6\}. \]
Then the corresponding Le diagram is constructed from the Young diagram fitting inside the $2\times(6-2)$ rectangle with rows labeled by $\{1,4\}$. Moreover we have:
\[\begin{cases}I_1\backslash I_2 =\{1\}\\ I_2\backslash I_1=\{2\}\end{cases} \;
\begin{cases}I_1\backslash I_3 =\{3\}\\ I_3\backslash I_1=\{1\}\end{cases} \;
\begin{cases}I_1\backslash I_4 =\{1\}\\ I_4\backslash I_1=\{6\}\end{cases} \;
\begin{cases}I_1\backslash I_5 =\{4>1\}\\ I_5\backslash I_1=\{5<6\}\end{cases} \;
\begin{cases}I_1\backslash I_6 =\{4\} \\ I_6\backslash I_1=\{6\} \end{cases}
\]
which means that we need to place a $+$ in the boxes $(1,2), (1,3), (1,6), (4,5), (4,6)$. Then the Le diagram associated to $\mB$ is the diagram $L$ on the left:
\begin{center}
    \begin{tikzpicture}[scale=0.75]
    \draw (0,0) --(0,2) --(4,2) --(4,1) --(2,1) --(2,0) --cycle;
    \draw (0,1) --(2,1) --(2,2);
    \draw (1,0) --(1,2);
    \draw (3,1) --(3,2);
    \draw (0.5,0.5) node{$+$};
    \draw (1.5,0.5) node{$+$};
    \draw (0.5,1.5) node{$+$};
    \draw (2.5,1.5) node{$+$};
    \draw (3.5,1.5) node{$+$};
    \draw (1.5,1.5) node{$0$};
    \draw (0.5,0) node[below]{$6$};
    \draw (1.5,0) node[below]{$5$};
    \draw (1.9,0.4) node[right]{$4$};
    \draw (2.6,.95) node[below]{$3$};
    \draw (3.5,.95) node[below]{$2$};
    \draw (4,1.5) node[right]{$1$};
    \draw (0,1) node[left]{$L = $};
    \end{tikzpicture}
\hspace{1.5cm} \begin{tikzpicture}[scale=0.75]
    \fill (0.5,0.5) circle(0.7mm);
    \fill (1.5,0.5) circle(0.7mm);
    \fill (0.5,1.5) circle(0.7mm);
    \fill (2.5,1.5) circle(0.7mm);
    \fill (3.5,1.5) circle(0.7mm);
    \fill (0.5,0) circle(0.7mm);
    \fill (1.5,0) circle(0.7mm);
    \fill (2.5,1) circle(0.7mm);
    \fill (2,0.5) circle(0.7mm);
    \fill (3.5,1) circle(0.7mm);
    \fill (4,1.5) circle(0.7mm);
    \draw[->] (3.9,1.5) --(3.6,1.5);
    \draw[->] (3.4,1.5) --(2.6,1.5);
    \draw[->] (2.4,1.5) --(0.6,1.5);
    \draw[->] (1.9,0.5) --(1.6,0.5);
    \draw[->] (1.4,0.5) --(0.6,0.5);
    \draw[->] (0.5,1.4) --(0.5,0.6);
    \draw[->] (0.5,0.4) --(0.5,0.1);
    \draw[->] (1.5,0.4) --(1.5,0.1);
    \draw[->] (2.5,1.4) --(2.5,1.1);
    \draw[->] (3.5,1.4) --(3.5,1.1);
    \draw (0.5,0) node[below]{$6$};
    \draw (1.5,0) node[below]{$5$};
    \draw (1.9,0.4) node[right]{$4$};
    \draw (2.6,.95) node[below]{$3$};
    \draw (3.5,.95) node[below]{$2$};
    \draw (4,1.5) node[right]{$1$};
    \draw (0,1) node[left]{$\Gamma(L) = $};
    \end{tikzpicture}\end{center}
Note that this is indeed a Le diagram. Moreover, we can easily see that the dimension of the corresponding positroid cell is equal to $5$, that is the number of $+$ in the Le diagram.

Let us now apply the algorithm to obtain the positroid associated to $L$ and check that this is in fact the positroid $\mB$. Note that the associated graph $\Gamma(L)$ is the diagram on the right.
Moreover, the $2$-subsets $B$ such that $(\{1,4\}\backslash B, \{2,3,5,6\}\cap B)$ admits a vertex-disjoint path system 
are:  $\{\{1,4\}, \{1,5\},\{1,6\}, \{2,4\},\{2,5\},\{2,6\},\{3,4\},\{3,5\},\{3,6\},\{4,6\},\{5,6\}\}$, which equals $\mB$. \demo
\end{example}

\noindent{\bf Positive reduced subexpressions.}
There is another way to understand positroids in terms of the Coxeter group. This section provides only a brief overview of the topic; for more details, see \cite{marsh2004parametrizations, williams2008total, knutson2013positroid}. Let $S_n$ be the symmetric group on $n$ elements, and let $v \in S_n$ be a permutation. Let $s_i = (i, i+1)$ denote the Coxeter generators of $S_n$. We can write an expression for $v$ in terms of the $s_i$ as $\mathbf{v} = v_1 v_2 \cdots v_m$. In general, we use bold letters for expressions and italic letters for permutations.

\medskip

Let $v_{(i)}$ be the product of the first $i$ letters in an expression $\mathbf{v}$; that is, $v_{(0)} = \varepsilon$ and $v_{(m)} = v$. Let $\ell(v)$ denote the length of $v$, i.e., the minimal number of generators needed to write $v$. An expression is called \emph{reduced} if $\ell(v_{(i+1)}) = \ell(v_{(i)}) + 1$ for all $i$. We say that $\mathbf{u}$ is a subexpression of $\mathbf{v}$ if $\mathbf{u}$ can be obtained by replacing certain letters of $\mathbf{v}$ with $\varepsilon$. There is a partial ordering on $S_n$ called the Bruhat order, where $u \leq v$ if and only if some expression for $u$ is a reduced subexpression of an expression for $v$. We call $[u,v]$ a Bruhat interval. A subexpression $\mathbf{u}$ of $\mathbf{v}$ is \emph{distinguished} if, whenever $\ell(u_{(i)} v_{i+1}) < \ell(u_{(i)})$, it follows that $v_{i+1} = v_i$. A distinguished subexpression is \emph{positive} if it is also reduced.

Positive subexpressions of reduced expressions are in one-to-one correspondence with Le diagrams \cite[Prop.~4.5]{williams2008total} (see also \cite{marcott2020combinatorics} for an exposition). Namely, the Bruhat order on $S_n$ induces a total order on $S_n / (S_k \times S_{n-k})$, where there is a unique Grassmannian representative $v$ such that $\bar{v} \in S_n / (S_k \times S_{n-k})$ defines a Ferrers shape, i.e., the shape of a Le diagram. The Ferrers shape gives rise to an expression $\mathbf{v}$ of $v$. Then the Le diagrams are in one-to-one correspondence with the positive subexpressions $\mathbf{u}$ of $\mathbf{v}$. Specifically, one can read off the positions of the $+$’s and $0$’s in the Le diagram defined by $(\mathbf{u}, \mathbf{v})$.

For example, in the Le diagram of Example~\ref{ex:pos}, the corresponding reduced expression is
$\mathbf{v} = s_4 s_5 s_1 s_2 s_3 s_4$,
while the positive subexpression is
$\mathbf{u} = \varepsilon \, \varepsilon \, \varepsilon \, \varepsilon \, s_3 \, \varepsilon$.
Note that the $\varepsilon$’s correspond to the placement of the $+$’s in the diagram, while the $s_i$ correspond to the $0$’s.

\medskip
The advantage of considering subexpression–expression pairs is that they provide useful insight into which positroids are contained in one another. Namely, if $\mathbf{u}$ and $\mathbf{u}'$ are two positive subexpressions of $\mathbf{v}$ such that $\mathbf{u}$ is a subexpression of $\mathbf{u}'$, then $(\mathbf{u}, \mathbf{v})$ is a boundary positroid of $(\mathbf{u}', \mathbf{v})$ whenever ${u}' \leq {u}$ in the Bruhat order. Furthermore, unpublished work \cite{Towercommunication} gives an algorithm to identify all boundaries of a positroid cell $(\mathbf{u}, \mathbf{v})$ in terms of Bruhat intervals.

\section{Main results \label{sec:results}}
\subsection{Positroid cells in the Grassmannian of lines \label{sec:cells}} 
We now focus on the case $k=2$. 
Our goal in this paper is to completely answer Questions~\ref{question:matroid} and \ref{question:max}.

\medskip

Note that if there exists a matrix satisfying the conditions in Question~\ref{question:matroid}, i.e.~it has nonzero (respectively nonnegative) $2$-minors, then $\mD^c$ is a representable matroid (respectively a positroid).
When there is no such a matrix, that is $\mD^c$ is not a positroid, then we are interested to compute the maximal positroids in $\mD^c$, particularly answering Question~\ref{question:max}. In general, there are multiple possible definitions of maximal that one may want to consider such as the cardinality or the dimension of the positroid. We are in particular interested to compute the maximal collections of $2$-pairs $\mM\subseteq \mD^c$ such that $\mM$ is a matroid (or a positroid), but there exists no other matroid (or a positroid) $\mM'$ such that $\mM\subset \mM'\subseteq \mD^c$. 

\begin{definition}\label{def:max}
Given a subset $\mD \subset \binom{[n]}{2}$,  we define Mat$(\mD)$ to be the set of maximal matroids contained in $\mD^c$ and $\MPos(\mD)$ for the set of maximal positroids contained in $\mD^c$.
\end{definition}

In this paper, we are interested in computing Mat$(\mD)$ and $\MPos(\mD)$. 
Note that any $2\times n$ matrix $A$ of rank $2$ defines a (representable) matroid $\mM = ([n], \mB)$ with $|B| = 2$ for all $B \in \mB$. Recall from \S\ref{sec:matroids} that the dependent sets of $\mM$ are all non-empty subsets of $[n]$ that are not contained in any basis. In particular, all subsets with at least three elements are dependent. Moreover, any zero column $i$, that is a circuit of size one, gives rise to dependent pairs $\{i,j\}$ for all $j\in[n]\backslash\{i\}$. See Remark~\ref{rem:zero_columns}. Therefore, to classify representable matroids or positroids of rank $2$, we only consider the dependent sets of size $2$. This property does not generalize beyond $k=2$.

\medskip 

We now 
define a bijection between collections of 
$2$-subsets of $[n]$ and graphs on  
subsets of $[n]$. 
\begin{definition}\label{def:graphs}
Given a subset $\mD \subset \binom{[n]}{2}$, we define $T_{\mD}=\{ i \in [n] \mid \{i,j\} \in \mD \text{ for all } j \in [n]\backslash \{i\}\}$
and $G_{\mD}$ the graph with vertex set $[n]\backslash T_{\mD}$ and edge set $\{\{i,j\} \in \mD \mid i,j \in [n]\backslash T_{\mD}\}$. Conversely, to any graph $G =( [n]\backslash T, E)$ we associate the subset $\mD_G = E \cup \{\{i,j\} \mid i \in T, j \in [n]\backslash\{i\}\}$
of $\binom{[n]}{2}$. 
Given $T \subset [n]$, we denote $P_{n,T}$ for the polygon with vertex set $[n]\backslash T$. (See Example~\ref{ex:crossing} and Figure~\ref{fig:examples}).
\end{definition}

\begin{remark}\label{rem:zero_columns}
Note that when $\mD^c$ is a matroid, $T_\mD$ is the set of loops in $\mD^c$, or equivalently, it encodes the zero columns in a matrix representation of $\mD^c$. Graphically, if a vertex (a.k.a.~column) does not contribute to \textit{any} basis set, then the set of $k$ planes in $\mathbb{R}^n$ defined by this matrix is the same as the set of $k$ planes in $\mathbb{R}^{n-|T_\mD|}$. 
More precisely, suppose that $\mD$ is the dependent set of a positroid and that $T_\mD \neq \emptyset$. Then there exists a $k\times n$ matrix representative $A$ for $\mD^c$ such that the columns indexed by $i \in T_\mD$ are the zero vector. Let $A'$ be the $k \times ([n]-|T_\mD|)$ matrix obtained from $A$ by eliminating the columns indexed by $T_\mD$. This is a matrix representative for the matroid with ground set $[n]\backslash T_{\mD}$ and dependent set $\{\{i,j\}\in \mD \mid i,j \not\in T_\mD\}$. In particular, every such subset $\mD$ can be identified as the dependent set $\mD'$ of a matroid with ground set $[n]\backslash T_\mD$ and $T_{\mD'}=\emptyset$ with $G_\mD = G_{\mD'}$.
\end{remark}

A first property $\mathcal{D}$ must have to define a positroid is to represent a matroid. This corresponds to the following condition on the graph $G_{\mathcal{D}}$.

\begin{lemma}\label{lemma:completecomp}
$\mD^c$ is a matroid if and only if every connected component of $G_{\mD}$ is a complete graph.
\end{lemma}
\begin{proof}
Let $\mD^c$ be the bases of a matroid. We want to prove that if $\{i,j\}, \{j,k\}\in \mD$ for some $i,j,k\in [n]$, then $\{i,k\} \in \mD$ or $j \in T_\mD$. Suppose by contradiction $\{i,k\} \in \mD^c$ and $j \not \in T_\mD$. Then $\{j,l\} \in \mD^c$ for some $l \in [n]\backslash\{i,k\}$. By definition of matroid, since $j \in \{j,l\}\backslash \{i,k\}$, either $\{i,j\}$ or $\{j,k\}$ is a basis. This is a contradiction, hence either $\{i,k\}\in \mD$ or $j \in T_\mD$, that is $G_\mD$ has complete connected components.

Assume now that $G_\mD$ has complete connected components. To prove that $\mD^c$ is a bases set of a matroid, we need to show that $\{i,k\}$ or $\{i,l\}$ is in $\mD^c$, if $\{i,j\},\{k,l\}\in \mD^c$. Suppose by contradiction $\{i,k\},\{i,l\}\in \mD$. Then either $i \in T_\mD$ or $\{k,l\}\in \mD$. The latter is not possible since $\{k,l\}\in \mD^c$ and if $i \in T_\mD$, then $\{i,j\}\in \mD$. This is a contradiction with the assumption that $\{i,j\}\in \mD^c$. 
\end{proof}
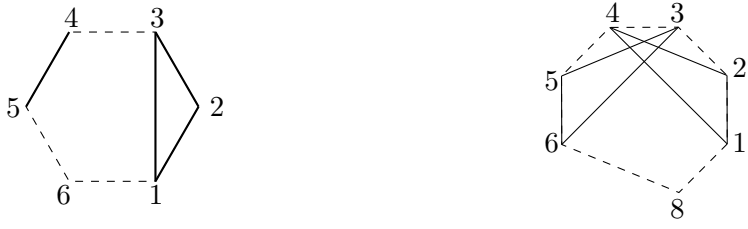
\begin{figure}[h]
    \centering
\begin{tikzpicture}[x=0.5pt,y=0.5pt,yscale=-1,xscale=1][h]
\draw  [thin,dashed] (230.5,186.25) -- (197.88,242.76) -- (132.63,242.76) -- (100,186.25) -- (132.62,129.74) -- (197.88,129.74) -- cycle ;
\draw [thick]    (197.88,129.74) -- (230.5,186.25) ;
\draw [thick]    (230.5,186.25) -- (197.88,242.76) ;
\draw [thick]    (197.88,129.74) -- (197.88,242.76) ;
\draw [thick]    (132.62,129.74) -- (100,186.25) ;
\draw (191,244) node [anchor=north west][inner sep=0.75pt]{$1$};
\draw (237,177.4) node [anchor=north west][inner sep=0.75pt]{$2$};
\draw (192,110) node [anchor=north west][inner sep=0.75pt]{$3$};
\draw (127,110) node [anchor=north west][inner sep=0.75pt]{$4$};
\draw (83,178.4) node [anchor=north west][inner sep=0.75pt]{$5$};
\draw (121,244) node [anchor=north west][inner sep=0.75pt]{$6$};
\draw  [dashed,thin] (630.43,214.46) -- (593.9,251.25) -- (505.25,214.9) -- (505.07,163.04) -- (541.6,126.25) -- (593.46,126.07) -- (630.25,162.6) -- cycle ;
\draw    (630.25,162.6) -- (630.43,214.46) ;
\draw    (541.6,126.25) -- (630.43,214.46) ;
\draw    (541.6,126.25) -- (630.25,162.6) ;
\draw    (505.07,163.04) -- (505.25,214.9) ;
\draw    (593.46,126.07) -- (505.25,214.9) ;
\draw    (505.07,163.04) -- (593.46,126.07) ;
\draw[thick] (632,204.4) node [anchor=north west][inner sep=0.75pt]    {$1$};
\draw[thick] (632,147.4) node [anchor=north west][inner sep=0.75pt]    {$2$};
\draw[thick] (585,106) node [anchor=north west][inner sep=0.75pt]    {$3$};
\draw[thick] (536,106) node [anchor=north west][inner sep=0.75pt]    {$4$};
\draw[thick] (490,157.4) node [anchor=north west][inner sep=0.75pt]    {$5$};
\draw[thick] (490,204.4) node [anchor=north west][inner sep=0.75pt]    {$6$};
\draw[thick] (585,254.4) node [anchor=north west][inner sep=0.75pt]    {$8$};
\end{tikzpicture}    \caption{On the right $G_\mD$ from Example~\ref{eq:small} and on the left $G_\mD$ from Example~\ref{ex:crossing}.}
    \label{fig:examples}
\end{figure}
\vspace{-5mm}
\begin{example}\label{eq:small}
Consider the positroid with bases set $\mathcal{B}$ from Example~\ref{ex:pos}. The corresponding dependent set is $\mathcal{D}=\mB^c=\{\{1,2\}, \{1,3\}, \{2,3\}, \{4,5\}\}$, $T_{\mD}=\emptyset$ and $G_{\mD}$ is the graph depicted in Figure~\ref{fig:examples}(left).
\demo
\end{example}

\begin{example} \label{ex:crossing}
Let $\mathcal{D}=\{\{1,2\},\{2,4\},\{1,4\},\{3,5\},\{3,6\},\{5,6\}\} \cup \{\{7,i\} \mid i \in [8]\backslash\{7\}\}.$
In this case, we see that $T_\mD = \{7\}$ and 
the graph $G_\mD$ can be represented as Figure~\ref{fig:examples}(right).
Note that in this case we ``forget" the vertex $7$ since every element of the form $\{7,i\}$ for $i \in [8],i \neq 7$ is in the set $\mathcal{D}$. The dashed heptagon in the picture represents the polygon $P_{8,\{7\}}$. 

\smallskip

We can see from the graph $G_\mD$ and Lemma~\ref{lemma:completecomp} that $\mD$ is the dependent set of a matroid. One way to determine whether $\mathcal{D}$ is the dependent set of a positroid is to apply the algorithm from \S\ref{sec:Le}, to construct the Le diagram from $\mD^c$ and check if the resulted diagram is a Le diagram, and whether it is indeed the 
Le diagram associated to $\mD^c$. If one of these two conditions is not satisfied, $\mD^c$ cannot be a positroid.
We have $I_1 = \{1,3\}$ and the Le diagram needs to fit inside a $2\times (8-2)$ rectangle. Constructing the associated diagram we obtain the diagram: 

\begin{figure}[h]
    \centering
    \begin{tikzpicture}[scale=0.7]
    \draw (0,0) --(0,2) --(6,2) --(6,1) --(5,1) --(5,0) --cycle;
    \draw (0,1) --(5,1) --(5,2);
    \draw (1,0) --(1,2);
    \draw (2,0) --(2,2);
    \draw (3,0) --(3,2);
    \draw (4,0) --(4,2);
    \draw (0.5,0) node[below]{$8$};
    \draw (1.5,0) node[below]{$7$};
    \draw (2.5,0) node[below]{$6$};
    \draw (3.5,0) node[below]{$5$};
    \draw (4.5,0) node[below]{$4$};
    \draw (5,0.4) node[right]{$3$};
    \draw (5.6,1) node[below]{$2$};
    \draw (6,1.5) node[right]{$1$};
    \draw (0.5,0.5) node{$+$};
    \draw (2.5,0.5) node{$+$};
    \draw (3.5,0.5) node{$+$};
    \draw (0.5,1.5) node{$+$};
    \draw (4.5,1.5) node{$+$};
    \draw (5.5,1.5) node{$+$};
    \draw (1.5,0.5) node{$0$};
    \draw (1.5,1.5) node{$0$};
    \draw (2.5,1.5) node{$0$};
    \draw (3.5,1.5) node{$0$};
    \draw (4.5,0.5) node{$0$};
    \end{tikzpicture} 
\end{figure}
However, this diagram is not a Le diagram, since the box in position $(3,4)$ contains $+$ both in the box over it in the same column and in some boxes to its left on the same row. Hence, $\mathcal{D}^c$ is not a positroid.
\demo\end{example}

\subsection{Combinatorial characterization of positroids
}\label{sec:niceanddim}
We now determine the necessary and sufficient conditions for a set 
to be the dependent set of a positroid.

\smallskip

Let us first fix our notation through this section. Here, $\mD$ is a subset of $\binom{[n]}{2}$.
\begin{notation} \label{notation:D,T,G}
\begin{itemize}
    \item If $T \subset [n]$ and $H=([n],E)$ is a graph, we denote by $H\backslash T$ the restriction of $H$ to $[n]\backslash T$, the edge set of which is given by $E(H\backslash T)= \{\{i,j\}\in E\mid i,j\in [n]\backslash T\}$.
    \item Given $\mD \subset \binom{[n]}{2}$, denote by $c(\mD)$ the number of connected components of the graph $G_\mD$. Note that every singleton vertex of $[n]\backslash T_\mD$ in $G_\mD$ is counted as one connected component.
    \item Finally, given a graph $H$, we denote by $\overline{H}$ the graph obtained from $H$ by adding all the necessary edges to $H$ to obtain a graph whose connected components are complete.
    Accordingly, given $\mathcal{D}\subset \binom{[n]}{2}$, we denote by $\overline{\mD}$ the subset of $\binom{[n]}{2}$ obtained from $\mD$ by adding every pair $\{i,k\}$ to $\overline{\mD}$ if $\{i,j\},\{j,k\}\in\mD$ with $j \not\in T_\mD$. In other words, 
    $\overline{\mD}= \mD \cup E(\overline{G_{\mD}})$.
\end{itemize}
\end{notation}

\begin{definition} \label{def:nice}
A subset $\mathcal{D}\subset \binom{[n]}{2}$ is called \textit{nice} on the cyclically ordered set $[n]$ if the following hold:
\begin{enumerate}
    \item[(1)] Every connected component of the associated graph $G_{\mD}$ is complete;
    \item[(2)] There is no connected component $C$ of $G_{\mathcal D}$ such that $P_{n,T_\mD}\backslash C$ is disconnected.
\end{enumerate}
\end{definition}

Note that the graph $G_{\mathcal D}$ from Example~\ref{ex:crossing} does not satisfy the second property above. For instance, if $C$ is the triangle on the vertices $\{3,5,6\}$, then $P_{8, \{7\}}\backslash C$ has two connected components: one consisting of the vertex $4$ and the other consisting of the part of the polygon comprised of vertices $\{1,2,8\}$. 
\begin{proposition}\label{prop:niceness}
$\mD^c$ is a positroid if and only if $\mD$ is nice.
\end{proposition}

\begin{proof}
Note that if $T_\mD \neq \emptyset$, by Remark~\ref{rem:zero_columns}, we can consider the set $\mD'=E(G_\mD) = \mD \backslash T_\mD$ with ground set $[n]\backslash T_\mD$, since $\mD^c$ is a positroid if and only if $(\mD')^c$ is a positroid. Moreover, $G_\mD = G_{\mD'}$ which implies that $\mD$ is nice if and only if $\mD'$ is nice. Therefore, we only need to consider the case $T_{\mD}=\emptyset$.

\medskip
First we show that if $\mD$ is not nice, then $\mD^c$ is not a positroid.  We may assume that every connected component of $G_\mD$ is complete. Otherwise, by Lemma~\ref{lemma:completecomp}, $\mD^c$ does not even define a matroid. 
Therefore, we only need to check 
the second condition of Definition~\ref{def:nice}. Let $C$ be a connected component of $G_\mD$.
Note that if $C$ has only one vertex, then $P_n\backslash C$ is connected. Moreover, we can assume that at least two of the vertices of $P_n$ are not in $C$, otherwise
$|P_n\backslash C| \in \{0,1\}$ and both the empty set and single vertex is connected.  In either case, this results in $\mD$ being nice. Therefore, it is sufficient to assume that there is a connected component of $G_\mD$ that consists of more than one, but not all vertices of $P_n$. Note that we can assume $1 \in V(C)$ and $n \not\in V(C)$. Indeed if this is not the case, by Remark~\ref{rmk:cyclicsym}, we can take a cyclic shift of the set $[n]$ so that this property is satisfied. 
Moreover, if $\mD'$ is the resulted set, then under this cyclic shift, $\mD^c$ is a positroid if and only if $(\mD')^c$ is a positroid. In other words, we may write the vertices of $C$ as $1 = v_1< v_2< \dots < v_k< n$. Suppose that $P_n\backslash C$ is disconnected and $C$ is complete. Then there exist some $i \in [n]$ and $j\in [k-1]$ such that $v_j<i<v_{j+1}$. Take the minimal such $i$.

With this assumption, we have that $I_1=\{1,i\}$ gives the row label of the Le diagram associated to $\mD^c$. We want to prove that if $(\mD')^c\subset \binom{[n]}{2}$ is the set associated to this diagram, then $\{1,v_{j+1}\}\in (\mD')^c$, from which it would follow that $\mD^c \neq (\mD')^c$, hence $\mD^c$ is not a positroid. Note that 
we have 
\[ I_1=\{1,i\}\text{ and } I_{v_{j+1}}=\{v_{j+1},k\} \text{ for some $k >v_{j+1}$.} \]
Thus, we need to add a $``+"$ in the box in position $(i,v_{j+1})$. Since there is a path from $i$ to $v_{j+1}$, following the algorithm in \S\ref{sec:Le} to construct a positroid from a Le diagram, we obtain a set $(\mD')^c$ containing the pair $\{1,v_{j+1}\}$. Since $\{1,v_{j+1}\}$ is an edge in $G_\mD$, it follows that $\{1,v_{j+1}\}\not\in \mD^c$, hence $\mD^c$ is not a positroid.

\medskip
Suppose now that $\mD$ is nice. 
Note that if one connected component of $G_\mD$ contains all the vertices $1,\ldots,n$, then $\mD^c =\emptyset$ is a positroid. We can now assume that no connected component $C$ of $G_\mD$ is such that $V(C)=[n]$.
Let $C_1$ be a connected component of $G_\mD$ and we can assume as before, up to relabelling the vertices with a cyclic rotation, that $V(C_1)=\{1,\dots, k\}$ for some $k<n$.
Note that $C_1$ must contain \emph{all} vertices between $1$ and $k$. Otherwise, 
$P_{[n]} \backslash C_1$ has two connected components,  one containing the missing vertex from $[k]$ and one containing $n$.
Since the component $C_1$ is complete, $I_1=\{1,k+1\}$ is the lexicographically minimal element in ${\mathcal{D}}^c$, hence $\{1,k+1\}$ is the column set in the diagram associated to $\mD^c$. Then, we fill the diagram as follows: insert a $``+"$ in every box in position $(1,i)$ for $i\in\{2,\dots,k\}$ and in position $(k+1,j)$ for $j \in \{k+2,\dots,n\}$. Moreover, for every $j\in \{k+1,\dots,n-1\}$, if $\{j,j+1\} \not\in \mD$, insert a $``+"$ in the box. Fill all the remaining empty boxes with a $``0"$. The resulted diagram is of form:
\begin{center}
    \begin{tikzpicture}[scale=0.5][h]
    \draw (0,0) --(0,2) --(8,2) --(8,1) --(4,1) --(4,0) --cycle;
    \draw (0,1) --(4,1) --(4,2);
    \draw (1,0) --(1,2);
    \draw (3,0) --(3,2);
    \draw (5,1) --(5,2);
    \draw (7,1) --(7,2);
    \draw (0.5,0) node[below]{$\textstyle n$};
    \draw (3.5,0) node[below]{$\textstyle k+2$};
    \draw (4,0.5) node[right]{$\textstyle k+1$};
    \draw (7.5,1) node[below]{$\textstyle 2$};
    \draw (8,1.5) node[right]{$\textstyle 1$};
    \draw (0.5,0.5) node{$+$};
    \draw (3.5,0.5) node{$+$};
    \draw (4.5,1.5) node{$+$};
    \draw (7.5,1.5) node{$+$};
    \draw (2,0.5) node{$\dots$};
    \draw (6,1.5) node{$\dots$};
    \end{tikzpicture}
\end{center}
Note that this is a Le diagram since $``0"$s only appear in the top row, satisfying the Le condition.

To check that such a Le diagram is indeed the Le diagram associated to $\mD^c$, we need to show that the diagram admits a vertex disjoint path $(\{1,k+1\}\backslash \{i,j\}, \{i,j\}\backslash\{1,k+1\})$ if and only if $\{i,j\} \in \mD^c$. Note that this occurs if and only if $i$ and $j$ are in different connected components of $G_\mD$. Assume without loss of generality that $i < j$. We may have the following cases: 

\smallskip
\noindent{\bf Case 1.} $1 \leq i, j \leq k$: Note that in this case $i, j \in V(C_1)$. In particular $\{i,j\}\in\mD$. On the other hand, there is no path in the Le diagram.  

\smallskip
\noindent{\bf Case 2.} $i \leq k$ and $j \geq k+1$: Then $i \in V(C_1)$ and $j \not \in V(C_1)$, in particular $\{i,j\}\in\mD^c$. Moreover, there is one path from $1$ turning down at $i$ and one path from $k+1$ turning down at $j$, as desired.

\smallskip
\noindent
\noindent{\bf Case 3.} $k+1 \leq i, j \leq n$: Then $\{i,j\}\in\mD^c$, or equivalently, $i$ and $j$ are in different connected components of $G_\mD$, if and only if there are two vertices $i \leq l < l+1 \leq j$ such that $l$, $l+1$ are in different connected components. In this case, by construction, there is one path from $k+1$ to $i$ and another path from $1$ to $j$ that turns down at $(1, l+1)$, left at $(k+1, l)$, and down at $(k+1, j)$. Conversely, there is no $``+"$ in the box $(1, l+1)$ for any $i < l+1 \leq j$, if and only if $i$ and $j$ are in the same connected component of $G_\mD$. Moreover, in this case any path from $\{1, k\}$ to $\{i, j\}$ must either share a vertex or a cross.
\end{proof}

\begin{remark}
Our characterization can be used to prove that every rank $2$ matroid is representable over $\mathbb{R}$. Recall that this is true for every rank $2$ matroids over any infinite field \cite{oxley2006matroid}. We show that for every matroid $\mM$ there is some (cyclic) order $\prec$ on the ground set $[n]$ such that with respect to that order $\mM$ is a positroid. Indeed, let $\mD$ be the dependent set of $\mM$ and let $C_1,\dots,C_k$ be the connected components of $G_\mD$. Then, since $V(C_1), \dots, V(C_k), T_\mD$ form a partition of $[n]$, we can define an order on $[n]$ such that for every $i<j$, the vertices of $C_i$ appear before the vertices of $C_j$, or equivalently $u < v$ for every $u \in V(C_i)$ and $ v \in V(C_j)$. Moreover, we can assume that the vertices in $T_\mD$ appear at the end. On the other hand, since $\mM$ is a matroid, the connected components of $G_\mD$ are all complete. Furthermore, with respect to this order, each $V(C_i)$ is a cyclic interval in $[n]\backslash T_\mD$, hence $\mD$ is a nice set. 
\medskip

Consider for example the matroid on the ground set $[10]$ whose dependent sets are given by \[\mD=\{\{1,2\},\{1,5\},\{2,5\},\{3,8\},\{3,9\},\{8,9\},\{4,7\},\{4,10\},\{7,10\}\}.\]
We note that $\mD$ is not a nice set with respect to the standard ordering of $[10]$. On the other hand if we define $\prec$ to be the cyclic ordering on $[10]$ defined by$ 1\prec 2\prec5\prec3\prec8\prec9\prec6\prec4\prec7\prec10 $ and denote by $\mD'$ the set $\mD$ considered with respect to the ordering $\prec$, then $\mD'$ is nice and $(\mD')^c$ defines a positroid.
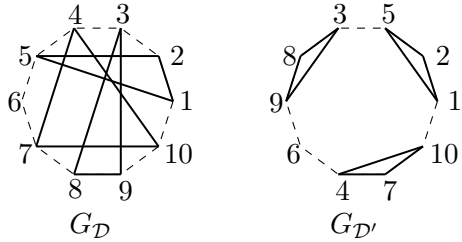
\begin{center}
\tikzset{every picture/.style={line width=0.75pt}} 
\begin{tikzpicture}[x=0.5pt,y=0.5pt,yscale=-1,xscale=1]
\draw  [dashed, thin] (196.5,135.99) -- (185.57,170.21) -- (156.94,191.35) -- (121.56,191.35) -- (92.93,170.21) -- (82,135.99) -- (92.93,101.78) -- (121.56,80.63) -- (156.94,80.63) -- (185.57,101.78) -- cycle ;
\draw    (92.93,101.78) -- (185.57,101.78) ;
\draw    (92.93,101.78) -- (196.5,135.99) ;
\draw    (185.57,101.78) -- (196.5,135.99) ;
\draw    (121.56,80.63) -- (92.93,170.21) ;
\draw    (92.93,170.21) -- (185.57,170.21) ;
\draw    (121.56,80.63) -- (185.57,170.21) ;
\draw    (121.56,191.35) -- (156.94,191.35) ;
\draw    (156.94,80.63) -- (121.56,191.35) ;
\draw    (156.94,80.63) -- (156.94,191.35) ;
\draw  [dashed, thin] (396.5,135.99) -- (385.57,170.21) -- (356.94,191.35) -- (321.56,191.35) -- (292.93,170.21) -- (282,135.99) -- (292.93,101.78) -- (321.56,80.63) -- (356.94,80.63) -- (385.57,101.78) -- cycle ;
\draw    (356.94,80.63) -- (385.57,101.78) ;
\draw    (356.94,80.63) -- (396.5,135.99) ;
\draw    (385.57,101.78) -- (396.5,135.99) ;
\draw    (321.56,80.63) -- (282,135.99) ;
\draw    (282,135.99) -- (292.93,101.78) ;
\draw    (321.56,80.63) -- (292.93,101.78) ;
\draw    (321.56,191.35) -- (356.94,191.35) ;
\draw    (385.57,170.21) -- (321.56,191.35) ;
\draw    (385.57,170.21) -- (356.94,191.35) ;
\draw (199,126.4) node [anchor=north west][inner sep=0.75pt]    {$1$};
\draw (192,92.4) node [anchor=north west][inner sep=0.75pt]    {$2$};
\draw (151,60.4) node [anchor=north west][inner sep=0.75pt]    {$3$};
\draw (115,61.4) node [anchor=north west][inner sep=0.75pt]    {$4$};
\draw (77,92.4) node [anchor=north west][inner sep=0.75pt]    {$5$};
\draw (69,127.4) node [anchor=north west][inner sep=0.75pt]    {$6$};
\draw (77,166.4) node [anchor=north west][inner sep=0.75pt]    {$7$};
\draw (115,193.4) node [anchor=north west][inner sep=0.75pt]    {$8$};
\draw (152.94,194.75) node [anchor=north west][inner sep=0.75pt]    {$9$};
\draw (187.57,166.61) node [anchor=north west][inner sep=0.75pt]    {$10$};
\draw (402,126.4) node [anchor=north west][inner sep=0.75pt]    {$1$};
\draw (393,92.4) node [anchor=north west][inner sep=0.75pt]    {$2$};
\draw (352,60.4) node [anchor=north west][inner sep=0.75pt]    {$5$};
\draw (315,60.4) node [anchor=north west][inner sep=0.75pt]    {$3$};
\draw (277,92.4) node [anchor=north west][inner sep=0.75pt]    {$8$};
\draw (266.94,127.75) node [anchor=north west][inner sep=0.75pt]    {$9$};
\draw (317,194.4) node [anchor=north west][inner sep=0.75pt]    {$4$};
\draw (353,194.4) node [anchor=north west][inner sep=0.75pt]    {$7$};
\draw (388.57,166.61) node [anchor=north west][inner sep=0.75pt]    {$10$};
\draw (280,166.4) node [anchor=north west][inner sep=0.75pt]    {$6$};
\draw (134, 230) node{$G_\mD$};
\draw (334, 230) node{$G_{\mD'}$};
\end{tikzpicture}
\end{center}
\end{remark}

\subsection{Maximal positroids
\label{sec:containment}}
Given a set $\mathcal{D}\subset \binom{[n]}{2}$, we can apply Proposition~\ref{prop:niceness} to determine whether $\mD^c$ is a positroid. When it is not the case, we may ask what are the maximal positroids contained in $\mD^c$, that is we wish to compute the set $\MPos(\mD)$.
There are, in general, different ways to shrink $\mD^c$ to a positroid, and we would like to compute the maximal (with respect to inclusion) such positroids. 
More precisely, we start by computing the maximal matroids contained in $\mD^c$. We then proceed by studying how to construct positroids lying inside a matroid. Furthermore, in Algorithm~\ref{alg:posT} we restrict these operations to compute all the maximal positroids. In particular, we answer Question~\ref{question:max} by studying the minimal dependency data that we need to add to a set $\mD$ so that the resulted set represents a positroid. 
\begin{example}\label{exam:3positroids}
Consider the set $\mD$ from Example~\ref{ex:crossing}, where $\mD^c$ is not a positroid, since the associated diagram in Example~\ref{ex:crossing} is not a Le diagram.
However, note that if the box $(1,4)$ was filled with a ``$0$" instead of a ``$+$", then we would obtain a Le diagram. This corresponds to adding the edge $\{3,4\}$ to the set $\mathcal{D}$. By Lemma~\ref{lemma:completecomp}, if a set $\mD'\supseteq	\mD\cup \{\{3,4\}\}$ represents the dependent sets of a positroid, 
then the connected components of its corresponding graph $G_{\mD'}$ should be complete. There are three ways to obtain such a graph from $G_{\mD}$ leading to three positroids of the following form:
\begin{itemize}
    \item $\mathcal{D}_1= \mathcal{D}\cup \{\{i,j\} \mid i \in \{1,2,4\},\ j\in\{3,5,6\}\}$, where $G_{\mD_1}$ is obtained from $G_\mD$ 
    by connecting the two crossing connected components;
    \item $\mathcal{D}_2 = \mathcal{D} \cup \{\{3,i\}\mid i \in [8]\backslash\{3\}\}$, where $G_{\mD_2}$ is obtained from $G_\mD$ by removing the vertex $3$;
    \item $\mathcal{D}_3 = \mathcal{D} \cup \{\{4,i\}\mid i \in [8]\backslash\{4\}\}$, where $G_{\mD_3}$ is obtained from $G_\mD$ by removing the vertex $4$; 
\end{itemize}
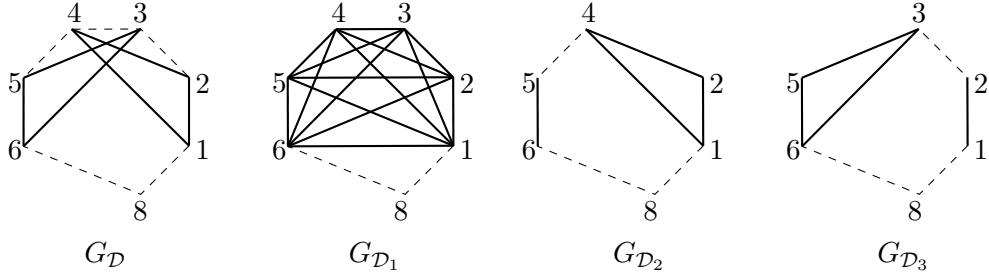
\begin{center}
\tikzset{every picture/.style={line width=0.75pt}}      
\begin{tikzpicture}[x=0.5pt,y=0.5pt,yscale=-1,xscale=1]

\draw  [dashed, thin] (141.43,209.46) -- (104.9,246.25) -- (16.25,209.9) -- (16.07,158.04) -- (52.6,121.25) -- (104.46,121.07) -- (141.25,157.6) -- cycle ;
\draw[thick]    (141.25,157.6) -- (141.43,209.46) ;
\draw[thick]    (52.6,121.25) -- (141.43,209.46) ;
\draw[thick]    (52.6,121.25) -- (141.25,157.6) ;
\draw[thick]    (16.07,158.04) -- (16.25,209.9) ;
\draw[thick]    (104.46,121.07) -- (16.25,209.9) ;
\draw[thick]    (16.07,158.04) -- (104.46,121.07) ;

\draw  [dashed,thin] (330.43,209.46) -- (293.9,246.25) -- (205.25,209.9) -- (205.07,158.04) -- (241.6,121.25) -- (330.25,157.6) -- cycle ;
\draw[thick]    (330.25,157.6) -- (330.43,209.46) ;
\draw[thick]    (241.6,121.25) -- (330.43,209.46) ;
\draw[thick]    (241.6,121.25) -- (330.25,157.6) ;
\draw[thick]    (205.07,158.04) -- (205.25,209.9) ;
 
\draw  [dashed, thin] (530.43,209.46) -- (493.9,246.25) -- (405.25,209.9) -- (405.07,158.04) -- (493.46,121.07) -- (530.25,157.6) -- cycle ;
\draw[thick]    (530.25,157.6) -- (530.43,209.46) ;
\draw[thick]    (405.07,158.04) -- (405.25,209.9) ;
\draw[thick]    (493.46,121.07) -- (405.25,209.9) ;
\draw[thick]    (405.07,158.04) -- (493.46,121.07) ;
\draw[thick]    (52.6,121.25) -- (104.46,121.07) ;
\draw[thick]    (52.6,121.25) -- (16.07,158.04) ;
\draw[thick]    (52.6,121.25) -- (16.25,209.9) ;
\draw[thick]    (104.46,121.07) -- (141.25,157.6) ;
\draw[thick]    (104.46,121.07) -- (141.43,209.46) ;
\draw[thick]    (16.25,209.9) -- (141.43,209.46) ;
\draw[thick]    (16.25,209.9) -- (141.25,157.6) ;
\draw[thick]    (16.07,158.04) -- (141.43,209.46) ;
\draw[thick]    (16.07,158.04) -- (141.25,157.6) ;

\draw  [dashed, thin] (-58.57,209.46) -- (-95.1,246.25) -- (-183.75,209.9) -- (-183.75,158.04) -- (-147.4,121.25) -- (-95.1,121.07) -- (-58.57,157.6) -- cycle ;
\draw[thick]    (-58.57,157.6) -- (-58.57,209.46) ;
\draw[thick]    (-147.4,121.25) -- (-58.57,209.46) ;
\draw[thick]    (-147.4,121.25) -- (-58.57,157.6) ;
\draw[thick]    (-183.75,158.04) -- (-183.75,209.9) ;
\draw[thick]    (-95.1,121.07) -- (-183.75,209.9) ;
\draw[thick]    (-183.75,158.04) -- (-95.1,121.07) ;

\draw (144,203.4) node [anchor=north west][inner sep=0.75pt]    {$1$};
\draw (144,153.4) node [anchor=north west][inner sep=0.75pt]    {$2$};
\draw (97,100) node [anchor=north west][inner sep=0.75pt]    {$3$};
\draw (47,100) node [anchor=north west][inner sep=0.75pt]    {$4$};
\draw (2,152.4) node [anchor=north west][inner sep=0.75pt]    {$5$};
\draw (2,203.4) node [anchor=north west][inner sep=0.75pt]    {$6$};
\draw (94,249.4) node [anchor=north west][inner sep=0.75pt]    {$8$};
\draw (333,203.4) node [anchor=north west][inner sep=0.75pt]    {$1$};
\draw (333,153.4) node [anchor=north west][inner sep=0.75pt]    {$2$};
\draw (236,100) node [anchor=north west][inner sep=0.75pt]    {$4$};
\draw (191,152.4) node [anchor=north west][inner sep=0.75pt]    {$5$};
\draw (191,203.4) node [anchor=north west][inner sep=0.75pt]    {$6$};
\draw (283,249.4) node [anchor=north west][inner sep=0.75pt]    {$8$};
\draw (533,203.4) node [anchor=north west][inner sep=0.75pt]    {$1$};
\draw (533,153.4) node [anchor=north west][inner sep=0.75pt]    {$2$};
\draw (486,100) node [anchor=north west][inner sep=0.75pt]    {$3$};
\draw (391,152.4) node [anchor=north west][inner sep=0.75pt]    {$5$};
\draw (391,203.4) node [anchor=north west][inner sep=0.75pt]    {$6$};
\draw (483,249.4) node [anchor=north west][inner sep=0.75pt]    {$8$};
\draw (-56,203.4) node [anchor=north west][inner sep=0.75pt]    {$1$};
\draw (-56,153.4) node [anchor=north west][inner sep=0.75pt]    {$2$};
\draw (-103,100) node [anchor=north west][inner sep=0.75pt]    {$3$};
\draw (-153,100) node [anchor=north west][inner sep=0.75pt]    {$4$};
\draw (-198,152.4) node [anchor=north west][inner sep=0.75pt]    {$5$};
\draw (-198,203.4) node [anchor=north west][inner sep=0.75pt]    {$6$};
\draw (-103,249.4) node [anchor=north west][inner sep=0.75pt]    {$8$};
\draw (60,280.4) node [anchor=north west][inner sep=0.75pt]    {$G_{\mD_1}$};
\draw (260,280.4) node [anchor=north west][inner sep=0.75pt]    {$G_{\mD_2}$};
\draw (460,280.4) node [anchor=north west][inner sep=0.75pt]    {$G_{\mD_3}$};
\draw (-140,280.4) node [anchor=north west][inner sep=0.75pt]    {$G_{\mD}$};
\end{tikzpicture}
\end{center}
As we see in the graphs above, the sets $\mathcal{D}_1, \mathcal{D}_2, \mathcal{D}_3$ are all nice, hence $\mD_1^c$, $\mD_2^c$, $\mD_3^c$ are
positroids. In particular, these positroids are all  contained in $\mD^c$.  \demo
\end{example}

 Since every positroid is in particular a matroid, we divide the process of finding the positroids contained in a given set $\mD^c$ into two steps: First, we find all the maximal matroids contained in $\mD^c$. Second, we determine all the maximal positroids contained in such matroids. 

\begin{notation}\label{notation:matroidprop}
\begin{itemize}
    \item Given $\mD \subset \binom{[n]}{2}$ and $T\subset [n]$, we define the set $\mD+ T$ as: 
\[ \mD + T := \mD \cup \{\{i,j\}\mid i \in T, j \in [n]\backslash\{i\}\}. \]
Note that we can assume $T \subset [n]\backslash T_\mD$. We also have $T_{\mD+T} = T_{\mD}\cup T$. 
\item Given $\mD \subset \binom{[n]}{2}$ we define the set $\mathbb{T}_\mD$ as the family of subsets $T$ of the vertices in $[n]\backslash T_\mD$ such that $c(\mD+T)$ is strictly larger than $c(\mD+S)$ for every subset $S \subsetneq T$. 
\[
\mathbb{T}_\mD:=\{T|\ T = \emptyset\text{ or }T\subset [n]\backslash T_{\mD}\text{ such that }c(\mD + T) > c(\mD + (T\backslash \{i\})\text{ for every }i\in T\}.
\]
In the terminology of graph theory, $T \in \mathbb{T}_\mD$ if and only if every $i \in T$ is a cut-point of the graph $G_{\mD+(T\backslash\{i\})}$ which is obtained from $G_\mD$ by eliminating the vertices in $T\backslash\{i\}$.
\end{itemize}\end{notation}

Recall $\overline{\mD} = \mD \cup E(\overline{G_\mD})$ from Notation \ref{notation:D,T,G}. 
\begin{proposition} \label{prop:matT}
The family of maximal matroids contained in $\mD^c$ is~given~by:
\[ \Mat(\mD) =\left\{\left(\overline{\mD+T}\right)^c \mid T \in \mathbb{T}_\mD\right\}. \]
\begin{proof}
Note that by construction every element in $\Mat(\mD)$ is a matroid. Let us start by proving that each $\mathcal{F}^c \in \Mat(\mD)$ is one of the maximal matroids contained in $\mD^c$. Let $\mathcal{F} = \overline{\mD +T}$ for some $T \in \mathbb{T}_\mD$. Suppose by contradiction that there exists $\mathcal{F}'$ with $\mD \subset \mathcal{F'} \subsetneq \mathcal{F}$ such that $(\mathcal{F}')^c$ is a matroid. Then, by Lemma~\ref{lemma:completecomp}, the connected components of the graphs $G_{\mathcal{F}'}$ and $G_{\mathcal{F}}$ are complete. Since $\mathcal{F}' \neq \mathcal{F}$, then $T_\mD \subseteq T_{\mathcal{F}'} \subsetneq T_{\mathcal{F}} = T\cup T_\mD$. Recall that $T\in\mathbb{T}_\mD$. Therefore, by definition of $\mathbb{T}_\mD$, for any $i \in T \backslash T_{\mathcal{F}'}$, we have that $i \in [n] \setminus T_\mD$. Thus, $c(\mD + T) > c(\mD +(T\backslash \{i\})) \geq c(\mD + T_{\mathcal{F}'})$.  
Let $C$ be the connected component in $G_{\mD+T_{\mathcal{F}'}}$
containing $i$. Note that by the previous inequality 
if we remove $T\backslash T_{\mathcal{F}'}$ from $C$, $C$ will be divided in at least two components $C_1$ and $C_2$, such that $\overline{C_1}\cup \overline{C_2} \subset G_{\mathcal{F}}$. For any $j_1 \in V(C_1)$ and $j_2 \in V(C_2)$, we have that $j_1,j_2 \in V(C)$ and since $C$ is a connected component in $G_{\mathcal{F}'}$ which is complete, we must have $\{j_1,j_2\} \in \mathcal{F'}$. On the other hand, $j_1$ and $j_2$ lie in 
different connected components in $G_\mathcal{F}$ and $j_1, j_2 \not\in T_\mathcal{F}$, hence $\{j_1, j_2\} \not \in \mathcal{F}$. This contradicts with the assumption that $\mathcal{F}' \subseteq \mathcal{F}$, so the claim follows.
\medskip

It remains to prove that all the maximal matroids contained in $\mD^c$ are of this form. Let $\mathcal{F} \supset \mD$ be such that $\mathcal{F}^c$ is a matroid and there is no $\mD \subseteq \mathcal{F'} \subsetneq \mathcal{F}$ such that $(\mathcal{F}')^c$ is a matroid. Let $T = T_\mathcal{F}\backslash T_\mD$. It is enough to prove that $T \in \mathbb{T}_\mD$. If $T=\emptyset$, then $T \in \mathbb{T}_\mD$. Assume that $T \neq \emptyset$, and suppose by contradiction that there is  $i \in T$ such that $c(\mD+T) = c(\mD+(T\backslash\{i\}))$ (that is, that $T \not \in \mathbb{T}_\mD$). Let $\mathcal{F}' = \overline{\mD + (T\backslash\{i\})}$ and let $C'$ be the connected component of $G_{\mathcal{F}'}$ containing $i$. We want to show that $\mathcal{F}' \subsetneq \mathcal{F}$ and $(\mathcal{F}')^c$ is a matroid. Note that the latter follows directly from Lemma~\ref{lemma:completecomp}, as we took the completion of the connected components in the graph $G_{\mathcal{F}'}$. It remains to prove that $E(C') \subset \mathcal{F}$. For every edge of the form $\{i,k\} \in E(C')$ for some $k \in V(C')\backslash\{i\}$, we have that $\{i,k\} \in \mathcal{F}$ since $i \in T_{\mathcal{F}}$. Now let $\{j,k\} \in E(C')$ with $j,k \in V(C')\backslash\{i\}$. Since removing $i$ did not change the number of connected components, the vertices $j$ and $k$ must lie in the same connected component of $G_\mathcal{F}$. This connected component is complete by Lemma~\ref{lemma:completecomp}, hence $\{j,k\} \in \mathcal{F}$. This is in contradiction with the assumption that $\mathcal{F}^c$ is one of the maximal positroids contained in $\mD^c$. Thus $c(\mD+T) > c(\mD+ (T\backslash \{i\})$ for every $i\in T$, and so $T\in \mathbb{T}_\mD$.
\end{proof}
\end{proposition}

\begin{example}
Let $\mD = \{\{2,3\},\{2,4\},\{2,6\}, \{3,4\}\} \subset \binom{[6]}{2}$. Then $\mathbb{T}_\mD=\{\emptyset, \{2\}\}$ and $ \Mat(\mD)=\{\overline{\mD},\ \overline{\mD+\{2\}}\}$. In particular, the corresponding graphs are:
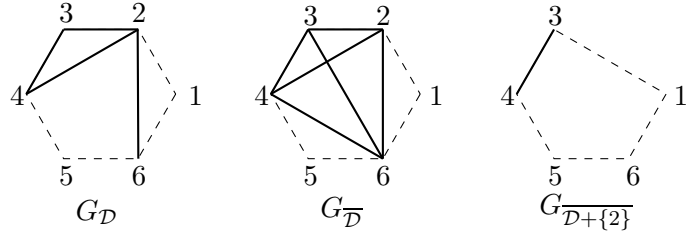
\begin{center}
\tikzset{every picture/.style={line width=0.75pt}}
\begin{tikzpicture}[x=0.5pt,y=0.5pt,yscale=-1,xscale=1]
\draw[dashed,thin] (-47.5,170.5) -- (-76,219.43) -- (-132.25,219.43) -- (-160.5,170.5) -- (-132.25,121.57) -- (-75.75,121.57) -- cycle ;
\draw[thick] (-132.25,121.57) -- (-75.75,121.57) ;
\draw[thick] (-132.25,121.57) -- (-160.5,170.5) ;
\draw[thick] (-160.5,170.5) -- (-75.75,121.57) ;
\draw[thick] (-76.25,121.57) -- (-75.75,219.43) ;
\draw (-40,162.4) node [anchor=north west][inner sep=0.75pt]    {$1$};
\draw (-83,100) node [anchor=north west][inner sep=0.75pt]    {$2$};
\draw (-138,100) node [anchor=north west][inner sep=0.75pt]    {$3$};
\draw (-175,162.4) node [anchor=north west][inner sep=0.75pt]    {$4$};
\draw (-138,223.4) node [anchor=north west][inner sep=0.75pt]    {$5$};
\draw (-83,223.4) node [anchor=north west][inner sep=0.75pt]    {$6$};
\draw (-107,260) node{$G_\mD$};

\draw  [dashed, thin] (137.5,170.5) -- (109.25,219.43) -- (52.75,219.43) -- (24.5,170.5) -- (52.75,121.57) -- (109.25,121.57) -- cycle ;
\draw[thick]    (52.75,121.57) -- (109.25,121.57) ;
\draw[thick]    (52.75,121.57) -- (24.5,170.5) ;
\draw[thick]    (24.5,170.5) -- (109.25,121.57) ;
\draw[thick]    (109.25,121.57) -- (109.25,219.43) ;
\draw[thick]    (24.5,170.5) -- (109.25,219.43) ;
\draw[thick]    (52.75,121.57) -- (109.25,219.43) ;

\draw  [dashed, thin] (323.5,170.5) -- (295.25,219.43) -- (238.75,219.43) -- (210.5,170.5) -- (238.75,121.57) -- cycle ;
\draw[thick]    (238.75,121.57) -- (210.5,170.5) ;
\draw (142,162.4) node [anchor=north west][inner sep=0.75pt]    {$1$};
\draw (102,102) node [anchor=north west][inner sep=0.75pt]    {$2$};
\draw (47,102) node [anchor=north west][inner sep=0.75pt]    {$3$};
\draw (10,162.4) node [anchor=north west][inner sep=0.75pt]    {$4$};
\draw (47,223.4) node [anchor=north west][inner sep=0.75pt]    {$5$};
\draw (102,223.4) node [anchor=north west][inner sep=0.75pt]    {$6$};
\draw (78,260) node{$G_{\overline{\mD}}$};
\draw (327,162.4) node [anchor=north west][inner sep=0.75pt]    {$1$};
\draw (232,102) node [anchor=north west][inner sep=0.75pt]    {$3$};
\draw (196,162.4) node [anchor=north west][inner sep=0.75pt]    {$4$};
\draw (232,223.4) node [anchor=north west][inner sep=0.75pt]    {$5$};
\draw (291,223.4) node [anchor=north west][inner sep=0.75pt]    {$6$};
\draw (263,260) node{$G_{\overline{\mD+\{2\}}}$};
\end{tikzpicture}
\end{center}
Note that $\mD + \{6\}$ is also the dependent set of a matroid, but $\overline{\mD}\subsetneq \mD+\{6\}$. \demo
\end{example}


From now on we may assume that $\mD$ is the dependent set of a matroid, i.e.~$\mD = \overline{\mD}$. Our goal is to provide an algorithm for computing $\MPos(\mD)$. 
To do so, we first give a combinatorial condition for the containment between dependent sets of a pair of matroids in terms of their associated graphs. 
\begin{lemma}\label{lemma:inclusion}
Let $\mD$ and $\mF$ be the dependent sets of two matroids. Then $\mD \subset \mF$ if and only if: \begin{itemize}
    \item[{\rm (a)}] $T_\mD \subset T_\mF$ and
    \item[{\rm (b)}] for every connected component $C$ of $G_\mD$ there exists a connected component $C'$ of $G_\mF$ such that $V(C)\backslash T_{\mF} \subset V(C')$.
\end{itemize}
\end{lemma}
\begin{proof}
Suppose $\mD \subset \mF$. Then (a) holds by definition. 
Moreover, every connected component $C$ of $G_\mD$ is complete, since $\mD^c$ is a matroid and $\{i,j\}\in \mD \subset \mF$ for every $i,j \in V(C)$. Then every pair of vertices $i,j \in V(C)\backslash T_{\mF}$ are connected in $G_\mF$ and so there exists a component $C'$ of $G_\mF$ with $V(C)\backslash T_\mF \subset V(C')$.
Suppose now that $\mD$ and $\mF$ are dependent sets of matroids such that (a) and (b) hold. 
Consider $\{i,j\}\in \mD$. If $i \in T_\mF$, then $\{i,j\}$ must be in $\mF$. If none of the indices $i,j$ are in $T_\mF$ then $i,j \in V(C)$ for some connected component $C$ of $G_\mD$, and by assumption there exists a connected component $C'$ of $G_\mF$ such that $i,j \in V(C)\backslash T_\mF \subset V(C')$. Since $C'$ is complete because $\mF$ is a matroid, then $\{i,j\}\in \mF$.
\end{proof}

Let us now fix our notation used in Algorithm~\ref{alg:posT} and throughout the rest of this section.

\begin{notation}\label{not:algorithm} 
\begin{itemize}
    \item For $\mathcal{D}\subset \binom{[n]}{2}$ and $C$ a connected component of $G_\mD$, we write
    \begin{align*}
        & P_{n,T_\mD}\backslash C = D_1 \sqcup \dots \sqcup D_k\quad\text{and}\quad C |_{P_{n,T_\mD}} = F_1 \sqcup \dots \sqcup F_k,
    \end{align*}
    where $D_1, \dots, D_k$ are the connected components of $P_{n,T_\mD}\backslash C$ and $F_1,\dots, F_k$ are those of
    $C |_{P_{n,T_\mD}}$. Note that each $D_i$ is a cyclic interval in $[n]\backslash(T_\mD\cup V(C))$ and $P_{n,T_\mD}\backslash C $ has the same number of connected components as $C |_{P_{n,T_\mD}}$. Moreover $P_{n,T_\mD} \backslash C$ is connected if and only if $k=1$. 
\item For any subgraph $D$ of $P_{n,T_\mD}$ and any connected component $C$ of $G_\mD$, we define 
\[C\cdot D=\{\{i,j\}\mid i \in V(C), j \in V(D)\}\]
which is obtained by connecting the vertices of $C$ to those of $D$. In the sequel, we will study~dependent sets of the form $\overline{\mD\cup (C\cdot D)}$ whose associated graphs are obtained from $G_\mD$ by connecting the vertices of $C$ to those of $D$ and then completing each connected component in the resulted graph.
\end{itemize}
\end{notation}

Note that in the notation above, $D$ can be \emph{any} subgraph of $P_{n,T_\mD}$, not just one of the connected components $D_i$. In particular, we may consider $D$ to either be one of the $D_i$ or a single vertex. Furthermore, to simplify our notation, we omit the closure and simply write $\mD\cup (C\cdot D)$.

\begin{lemma}\label{lem:twooperations}
Let $\mathcal{D}\subset\mathcal{D'} \subset \binom{[n]}{2}$ be dependent sets of two matroids such that $\mD^c$ is 
not a positroid and $(\mD')^c$ is a positroid. Then there exists a connected component $C$ of $G_\mD$ such that $P_{n,T_\mD}\backslash C$ is disconnected. Moreover there exists $\{i,j\} \in \mD'\backslash \mD$ for some $i \in V(C)$.
\end{lemma}

\begin{proof}
The first assertion follows by Proposition~\ref{prop:niceness} as $\mD^c$ is not a positroid. Let $C$ be such a connected component. 
Assume by contradiction that there is no $\{i,j\}$ in $\mD'\backslash\mD$ for some vertex $i\in V(C)$.  
Then we must have $T_\mD = T_{\mD'}$. Indeed, suppose by contradiction that there exists $j \in T_{\mD'}\backslash T_\mD$. If $j \in V(C)$ then since $j \not\in T_\mD$ there exists $k \in [n]\backslash (T_{\mD}\cup V(C))$ such that $\{j,k\}\not\in \mD$ but $\{j,k\}\in \mD'$. If $j \not\in V(C)$ then $\{i,j\}\in \mD'\backslash \mD$ for every $i \in V(C)$.
Moreover, $C$ is a connected component in $G_{\mD'}$ since, by assumption, no edge touching $C$ was added in $\mD'$. Hence $P_{n,T_{\mD'}} \backslash C = P_{n,T_\mD}\backslash C$, which is disconnected. This contradicts with the assumption that $(\mD')^c$ is a positroid, by Proposition~\ref{prop:niceness}.
\end{proof}
\begin{remark}\label{rmk:twooperations}
Given $\{i,j\}\in \mathcal{D}'\backslash \mathcal{D}$ as in Lemma~\ref{lem:twooperations}, 
we either have $\mD \cup (C\cdot\{j\}) \subseteq \mathcal{D'}$ or $i \in T_{\mathcal{D'}}$ by Lemma~\ref{lemma:completecomp}. Note that we do not need to consider the case $j \in T_{\mD'}$, since we are only interested in determining the maximal positroids and in this case $\mD \cup (C \cdot \{j\}) \subsetneq \mD+\{j\}$. Having only these two cases shows that we can construct a nice set $\mD' \supset \mD$ by performing a sequence of operations as follows: 
\begin{itemize}
    \item[(1)] Let $C$ be a connected component of $G_\mD$ such that $P_{n,T_\mD}\backslash C$ is disconnected. Connect $C$ to a vertex $j \in [n]\backslash (T_\mD \cup V(C))$ and add all the edges to obtain a complete component. Then we obtain the set $\mD \cup (C\cdot\{j\})$; 
    \item[(2)] Given a vertex $i \in V(C)$, where $C$ is a component of $G_\mD$ such that $P_{n,T_\mD}\backslash C$ is disconnected, add the vertex $i$ to the vanishing set $T_\mD$. Then we obtain the set $\mD+\{i\}$. Note that $C$ is no longer a connected component of $G_{\mD+\{i\}}$, but its restriction to the vertex set $V(C) \backslash \{i\}$ is.
\end{itemize}
\end{remark}
\begin{remark}
Since we are {\em only} interested to obtain the {\em maximal} positroids contained in a given matroid, we can further restrict the two operations from Remark~\ref{rmk:twooperations}
to the following two operations:
\begin{enumerate}[label=(\alph*)]
    \item\label{operation:connection} Given a connected component $C$  of $G_{\mD}$ such that $P_{n,T_\mD}\backslash C$ is disconnected, and a connected component $D$ of $P_{n,T_\mD}\backslash C$, we let $\mD'= \mD \cup (C\cdot D)$;
    \item\label{operation:elimination} Given a connected component $C$ of $G_{\mD}$ such that $P_{n,T_\mD}\backslash C$ is disconnected, and a connected component $F$ of $C|_{P_{n,T_\mD}}$, we let $\mD' = \mD+F$.
\end{enumerate}
If $\mF \supsetneq \mD$ is the dependent set of a positroid contained in the matroid $\mD^c$ where $\mF$ is constructed by connecting a connected component $C$ of $G_\mD$ to a component strictly contained in a connected component $D\subsetneq D_i$ of $P_{n,T_\mD}\backslash C$, then we can obtain a smaller set $\mF'$ with $\mF \supsetneq \mF' \supsetneq \mD$ such that $(\mF')^c$ is a positroid. Similarly, if $\mF \supsetneq \mD$ is the dependent set of a positroid and $\mF$ is constructed by enlarging 
the vanishing set $T$ by adding a component strictly contained in a connected component $F \subsetneq F_i$ of $C|_{P_{n,T_\mD}}$ to $T$, we obtain a smaller set $\mF'$ with $\mF \supsetneq \mF' \supsetneq \mD$ such that $(\mF')^c$ is a positroid. This will be proved in Proposition~\ref{prop:posT}.
\end{remark}
\begin{example}
Consider the matroid $\mD^c$ where $\mD=\{\{1,2\},\{1,5\},\{2,5\}\}\subset \binom{[6]}{2}$. If $C$ is the connected component of $G_\mD$ given by the triangle on vertices $\{1,2,5\}$ then we have that $P_6\backslash C = D_1 \cup D_2$ with $V(D_1)= \{3,4\}$, $V(D_2)=\{6\}$ and $C|_{P_6} = F_1 \cup F_2$ with $V(F_1)=\{1,2\}$ and $V(F_2)=\{5\}$.

Let $\mF$ be the dependent set of a positroid and suppose $\mF'=\mD \cup (C\cdot\{3\}) \subset \mF$ but $\{4,i\}\not\in \mF$ for $i=1,2,5$. Then there are two possibilities to obtain a positroid from $\mF'$ by either connecting the connected component $C' = C\cdot\{3\}$ to the vertex $6$ or adding the vertex $5$ to the vanishing set. 
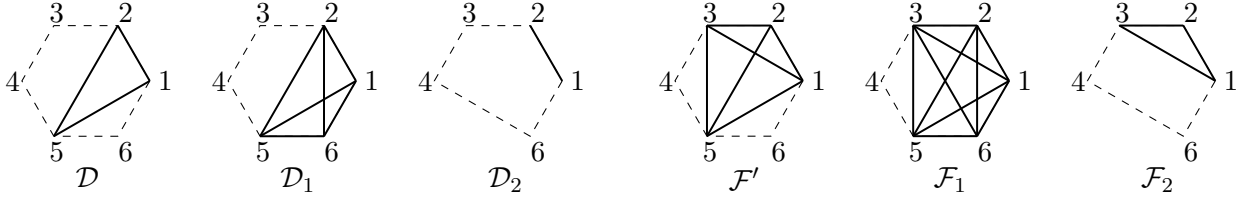
\begin{center}
\tikzset{every picture/.style={line width=0.75pt}} 

\begin{tikzpicture}[x=0.52pt,y=0.52pt,yscale=-1,xscale=1]

\draw[dashed, thin]   (158,62.5) -- (134.75,102.77) -- (88.25,102.77) -- (65,62.5) -- (88.25,22.23) -- (134.75,22.23) -- cycle ;
\draw[dashed, thin]  (308,62.5) -- (284.75,102.77) -- (238.25,102.77) -- (215,62.5) -- (238.25,22.23) -- (284.75,22.23) -- cycle ;
\draw[dashed, thin]   (458,62.5) -- (434.75,102.77) -- (365,62.5)-- (388.25,22.23) -- (434.75,22.23) -- cycle ;

\draw    (158,62.5) -- (134.75,22.23) ;
\draw    (134.75,22.23) -- (88.25,102.77) ;
\draw    (88.25,102.77) -- (158,62.5) ;
\draw    (284.75,22.23) -- (238.25,102.77) ;
\draw    (284.75,22.23) -- (308,62.5) ;
\draw    (238.25,102.77) -- (308,62.5) ;
\draw    (284.75,102.77) -- (238.25,102.77) ;
\draw    (284.75,102.77) -- (308,62.5) ;
\draw    (284.75,22.23) -- (284.75,102.77) ;
\draw    (434.75,22.23) -- (458,62.5) ;

\draw (162,53.4) node [anchor=north west][inner sep=0.75pt]    {$1$};
\draw (133,4.4) node [anchor=north west][inner sep=0.75pt]    {$2$};
\draw (83,4.4) node [anchor=north west][inner sep=0.75pt]    {$3$};
\draw (51,54.4) node [anchor=north west][inner sep=0.75pt]    {$4$};
\draw (83,105.4) node [anchor=north west][inner sep=0.75pt]    {$5$};
\draw (133,105.4) node [anchor=north west][inner sep=0.75pt]    {$6$};
\draw (312,53.4) node [anchor=north west][inner sep=0.75pt]    {$1$};
\draw (283,4.4) node [anchor=north west][inner sep=0.75pt]    {$2$};
\draw (233,4.4) node [anchor=north west][inner sep=0.75pt]    {$3$};
\draw (201,54.4) node [anchor=north west][inner sep=0.75pt]    {$4$};
\draw (233,105.4) node [anchor=north west][inner sep=0.75pt]    {$5$};
\draw (283,105.4) node [anchor=north west][inner sep=0.75pt]    {$6$};
\draw (462,53.4) node [anchor=north west][inner sep=0.75pt]    {$1$};
\draw (433,4.4) node [anchor=north west][inner sep=0.75pt]    {$2$};
\draw (383,4.4) node [anchor=north west][inner sep=0.75pt]    {$3$};
\draw (351,54.4) node [anchor=north west][inner sep=0.75pt]    {$4$};
\draw (433,105.4) node [anchor=north west][inner sep=0.75pt]    {$6$};
\draw (102,123.4) node [anchor=north west][inner sep=0.75pt]    {$\mathcal{D}$};
\draw (251,123.4) node [anchor=north west][inner sep=0.75pt]    {$\mathcal{D}_{1}$};
\draw (402,123.4) node [anchor=north west][inner sep=0.75pt]    {$\mathcal{D}_{2}$};
\end{tikzpicture} 
\qquad
\begin{tikzpicture}[x=0.52pt,y=0.52pt,yscale=-1,xscale=1]
\draw[dashed, thin]   (158,222.5) -- (134.75,262.77) -- (88.25,262.77) -- (65,222.5) -- (88.25,182.23) -- (134.75,182.23) -- cycle ;
\draw[dashed, thin]   (308,222.5) -- (284.75,262.77) -- (238.25,262.77) -- (215,222.5) -- (238.25,182.23) -- (284.75,182.23) -- cycle ;
\draw[dashed, thin]   (458,222.5) -- (434.75,262.77) -- (365,222.5) -- (388.25,182.23) -- (434.75,182.23) -- cycle ;
\draw    (134.75,182.23) -- (158,222.5) ;
\draw    (134.75,182.23) -- (88.25,262.77) ;
\draw    (88.25,262.77) -- (158,222.5) ;
\draw    (88.25,182.23) -- (88.25,262.77) ;
\draw    (88.25,182.23) -- (134.75,182.23) ;
\draw    (88.25,182.23) -- (158,222.5) ;
\draw    (238.25,182.23) -- (284.75,182.23) ;
\draw    (284.75,182.23) -- (308,222.5) ;
\draw    (238.25,182.23) -- (308,222.5) ;
\draw    (308,222.5) -- (284.75,262.77) ;
\draw    (284.75,262.77) -- (238.25,262.77) ;
\draw    (238.25,182.23) -- (238.25,262.77) ;
\draw    (238.25,182.23) -- (284.75,262.77) ;
\draw    (284.75,182.23) -- (238.25,262.77) ;
\draw    (284.75,182.23) -- (284.75,262.77) ;
\draw    (238.25,262.77) -- (308,222.5) ;
\draw    (388.25,182.23) -- (434.75,182.23) ;
\draw    (434.75,182.23) -- (458,222.5) ;
\draw    (388.25,182.23) -- (458,222.5) ;
\draw (162,213.4) node [anchor=north west][inner sep=0.75pt]    {$1$};
\draw (133,164.4) node [anchor=north west][inner sep=0.75pt]    {$2$};
\draw (83,164.4) node [anchor=north west][inner sep=0.75pt]    {$3$};
\draw (51,214.4) node [anchor=north west][inner sep=0.75pt]    {$4$};
\draw (83,265.4) node [anchor=north west][inner sep=0.75pt]    {$5$};
\draw (133,265.4) node [anchor=north west][inner sep=0.75pt]    {$6$};
\draw (312,213.4) node [anchor=north west][inner sep=0.75pt]    {$1$};
\draw (283,164.4) node [anchor=north west][inner sep=0.75pt]    {$2$};
\draw (233,164.4) node [anchor=north west][inner sep=0.75pt]    {$3$};
\draw (201,214.4) node [anchor=north west][inner sep=0.75pt]    {$4$};
\draw (233,265.4) node [anchor=north west][inner sep=0.75pt]    {$5$};
\draw (283,265.4) node [anchor=north west][inner sep=0.75pt]    {$6$};
\draw (462,213.4) node [anchor=north west][inner sep=0.75pt]    {$1$};
\draw (433,164.4) node [anchor=north west][inner sep=0.75pt]    {$2$};
\draw (383,164.4) node [anchor=north west][inner sep=0.75pt]    {$3$};
\draw (351,214.4) node [anchor=north west][inner sep=0.75pt]    {$4$};
\draw (433,265.4) node [anchor=north west][inner sep=0.75pt]    {$6$};
\draw (102,283.4) node [anchor=north west][inner sep=0.75pt]    {$\mathcal{F} '$};
\draw (251,283.4) node [anchor=north west][inner sep=0.75pt]    {$\mathcal{F}_{1}$};
\draw (402,283.4) node [anchor=north west][inner sep=0.75pt]    {$\mathcal{F}_{2}$};
\end{tikzpicture}
\end{center}
We obtain respectively $\mF_1 = \mF' \cup (C' \cdot \{6\})$ and $\mF_2 = \mF'\cdot\{5\}$. Let $\mD_1, \mD_2$ be the sets obtained by performing the same operations but not connecting $C$ to $\{3\}$, that is $\mD_1 = \mD \cup (C\cdot\{6\})$ and $\mD_2 = \mD \cdot \{5\}$. It is easy to see that $\mD_1 \subsetneq \mF_1$ and $\mD_2 \subsetneq \mF_2$, and the associated graphs are the ones depicted above.\demo
\end{example}
\noindent{\bf Description of Algorithm~\ref{alg:posT}.} We use the set $\mathcal{G}$ to keep track of the intermediate sets that do not define a positroid. We denote $\mathcal{C}$ for the set of connected components of the graph $G_\mD$. We start with $\mD$ and check whether it is nice, i.e.~if the vertices of the connected components of $G_\mD$ lie on cyclic intervals of $[n]\setminus T_\mD$. If it is nice, then the algorithm terminates. Otherwise, for each connected component $C \in \mathcal{C}$, for which the vertices do not define a cyclic interval, it considers the connected components of $P_{n, T_\mD} \setminus C = D_1 \sqcup\cdots\sqcup D_k$ and the cyclic intervals defined by $C$, i.e.~ $C|_{P_{n, T_\mD}} = F_1 \sqcup\cdots\sqcup F_k$. The algorithm then identifies each dependent set formed by completing the graph on $V(C) \cup D_i$ (denoted $\mD \cup (C \cdot D_i)$) or connecting the vertices of $F_i$ to every vertex in $[n] \setminus T_\mD$ (denoted $\mD + F_i$), and then repeats the process for each set. 

\medskip

We now prove that the output of Algorithm~\ref{alg:posT}, i.e.~the set $\Pos(\mD)$
contains all the maximal positroids contained in a given matroid $\mD^c$. First we recall from Definition~\ref{def:max} that 
\[\MPos(\mD)=\{\mB |\ \mB^c \textrm{ is a positroid and there exist no } 
\mD'  \textrm{ s.t. }\mB \supsetneq \mD' \supseteq \mD \textrm{ and } (\mD')^c \textrm{ is a positroid}\}.\]

\begin{proposition} \label{prop:posT}
Let $\mD \subset \binom{[n]}{2}$ be such that $\mD^c$ is a matroid. Then the set $\Pos(\mD)$ obtained by applying 
Algorithm~\ref{alg:posT} to $\mD$ 
contains $\MPos(\mD)$. 
\end{proposition}
\begin{proof}
First note that if $\mD$ is nice, then $\mD^c$ is a positroid by Proposition~\ref{prop:niceness}, and so it is the unique maximal positroid that we are looking for. In this case, $\Pos(\mD) = \{\mD^c\}$ which is equal to $\MPos(\mD)$.

Suppose now that $\mD$ is not nice. Consider an element $\mD'\supset\mD$ in $\MPos(\mD)$.
We want to show~that $\mD' \in \Pos(\mD)$. Since $\mD^c$ is a matroid, every connected component of $G_\mD$ is complete. By Lemma~\ref{lem:twooperations}, there exists a connected component $C$ of $G_\mD$ such that $P_{n,T_\mD}\backslash C$ is disconnected, and there exists $\{i,j\} \in \mD' \backslash \mD$ for some $i \in V(C)$. Then, by Remark~\ref{rmk:twooperations} there are two possible cases that we consider separately.

\medskip\noindent{\bf Case 1.} $\mD \cup (C \cdot\{j\}) \subset \mD'$. Let $D$ be the connected component of $P_{n,T_\mD}\backslash C$ containing $j$. We need to prove that $\mD \cup (C\cdot D)$ is contained in $\mD'$. Suppose by contradiction that there exists $k \in V(D)$ such that $\{k,\ell\} \not\in \mD'$ for some $\ell \in V(C)$. Define the sets
\[ D(G_{\mD'})=\{j \in V(D) \mid \{i,j\}\in \mD' \text{ for all } i \in V(C)\}\text{ and } \mF=\mD' \backslash \{\{i,\ell\}\mid i \in V(C), \ell \in D(G_{\mD'})\}. \]
Note that $\overline{\mF} \subseteq \mD'$. We show that $\overline{\mF}$ is nice. Suppose by contradiction that $\overline{\mF}$ is not nice. Then the connected component $C'$ in $G_{\overline{\mF}}$ containing $V(C)\backslash T_{\mD'}$ is such that $P_{n,T_{\mF}}\backslash C' = D_1 \sqcup D_2$ where $D_1$ is the connected component containing $D(G_{\mD'})$. Since the changes in the construction of $\mF$ only involve the vertices in $D_1$, it follows that $\{i,\ell\} \in \mD'$ for all $i \in V(C')$ and $\ell \in D_1$. Hence $\{i,\ell\} \in \mD'$ for all $i \in V(C)$ and $\ell \in V(D)$, which is a contradiction. Therefore $\overline{\mF}$ is nice.

\medskip\noindent{\bf Case 2.} $i \in T_{\mD'}$. Let $F$ be the connected component of $C|_{P_{n,T_\mD}}$ containing $i$. We prove that $\mD + F \subset \mD'$. Suppose by contradiction that $V(F)\nsubseteq T_{\mD'}$ 
and let $k \in V(F) \backslash T_{\mD'}$. Define 
\[
\mF=\mD' \backslash \{\{i,\ell\}\mid i \in V(F)\cap T_{\mD'}, \ell \in [n]\backslash\{i\}\} \cup \{\{i,k\} \mid i \in V(F)\cap T_{\mD'}, k \in V(C)\backslash\{i\}\} \;.\] 
Note that $\mF$ is defined by removing all the elements of $V(F) \cap T_{\mD'}$ from the vanishing set.
Now let $\overline{\mF}$ be the set obtained from $\mF$ by completing the connected components of $G_\mF$. If we prove that $\overline{\mF}$ is nice, since $\overline{\mF}\subseteq \mD'$, we would get a contradiction with the assumption that $\mD'$ is one of the smallest set (containing $\mD$) with the required property. Let $C'$ be the connected component in $G_{\overline{\mF}}$ containing $k$. If $P_{n,T_{\mD'}}\backslash C'$ is not connected, since the only changes made in the process of constructing $\mF$ involve the vertices of $F$, there is a connected component of $P_{n,T_{\mD'}}\backslash C'$ contained in $V(F)\backslash V(C')$. By definition of $\mF$, the set $V(F) \backslash V(C')$ is empty, hence $\overline{\mF}$ is nice.
\end{proof}
\begin{algorithm}[H]
\caption{Finding all the maximal positroids contained in a given matroid.
}
\label{alg:posT}
\KwIn{A subset $\mD \subset \binom{[n]}{2}$ such that $\mD^c$ is a matroid.
\\
}
\BlankLine
\KwOut{A set containing all the maximal positroids contained in the matroid $\mD^c$.}
\BlankLine
{\bf Initialization:}
$\mathcal{G} \leftarrow \{ \mD \}$\;
$\mathcal{C} \leftarrow \emptyset$\;
$\Pos(\mD) \leftarrow \emptyset$\;
\BlankLine
\While{ $\mathcal{G} \neq \emptyset$}
{
\ForEach{$\mathcal{F} \in \mathcal{G}$}
{
\BlankLine
$\mathcal{C} \leftarrow \{ C_1,\dots,C_m \mid G_\mathcal{F} = C_1\cup \dots \cup C_m\}$\;
\eIf{ for all $i \in [m]$, $P_{n,T_\mathcal{F}}\backslash C_i$ is connected}
{
\BlankLine
$\Pos(\mD) \leftarrow \Pos(\mD) \cup \{\mathcal{F}^c\}$\;
$\mathcal{G} \leftarrow \mathcal{G} \backslash \{\mathcal{F} \}$\;
}
{
$P_{n,T_\mathcal{F}}\backslash C_i = D^i_1\cup \dots\cup D^i_{k_i}$, for all $i \in [m]$\;
$C_i|_{P_{n,T_\mathcal{F}}} = F^i_1 \cup \dots \cup F^i_{k_i}$, for all $i \in [m]$\;
\ForEach{ $j \in [m]$ such that $k_j >1$}
{
$\mathcal{G} \leftarrow \mathcal{G} \cup \{ \mathcal{F} \cup (C_j \cdot D^j_1),\, \dots,\, \mathcal{F} \cup (C_j \cdot D^j_{k_j}),\, \mathcal{F} + F^j_1,\, \dots,\, \mathcal{F} +F^j_{k_j}\}$\;
}
$\mathcal{G} \leftarrow \mathcal{G} \backslash \{\mathcal{F}\}$\;
}
}
}
\Return{$\Pos(\mD)$}\;
\end{algorithm}

\medskip

\begin{example}\label{ex:algorithm}
Consider $\mathcal{D}=\{\{1,2\},\{2,4\},\{1,4\},\{3,5\},\{3,6\},\{5,6\}\} \cup \{\{7,i\} \mid i \in [8]\backslash\{7\}\}$ from Examples~\ref{ex:crossing} and \ref{exam:3positroids}.
Note that $\mD$ is a matroid, since $G_\mD$ has complete connected components, but it is not a positroid. We want to construct the set $\Pos(\mD)$ applying Algorithm~\ref{alg:posT}. Let $\mathcal{G}=\{\mD\}$. Consider the graph $G_\mD$ and let $C_1$ be the connected component with vertices $\{1,2,4\}$ and let $C_2$ be the connected component with vertices $\{3,5,6\}$. Let us start by considering the connected component $C_1$. We have
\[ P_{8,\{7\}} \backslash C_1 = D^1_1 \cup D^1_2, \quad C_1|_{P_{8,\{7\}}} = F^1_1 \cup F^1_2 \]
with $V(D^1_1) = \{3\}$, $V(D^1_2)=\{5,6,8\}$ and $V(F^1_1)=\{1,2\}$, $V(F^1_2)=\{4\}$. Applying the algorithm to this connected component we obtain the sets with associated graphs:

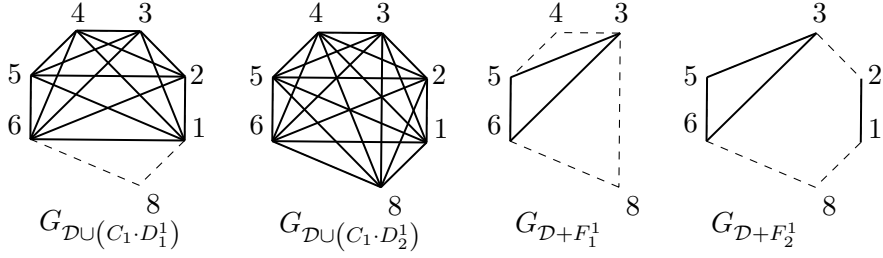
\begin{center}
\tikzset{every picture/.style={line width=0.75pt}} 
\begin{tikzpicture}[x=0.75pt,y=0.75pt,yscale=-1,xscale=1]

\draw  [dashed, thin] (115.88,177.54) -- (92.89,200.12) -- (38.08,176.86) -- (38.36,144.63) -- (61.34,122.04) -- (93.57,122.33) -- (116.16,145.31) -- cycle ;
\draw  [dashed, thin] (237.88,178.54) -- (214.89,201.12) -- (160.08,177.86) -- (160.36,145.63) -- (183.34,123.04) -- (215.57,123.33) -- (238.16,146.31) -- cycle ;
\draw  [dashed, thin] (334.89,201.12) -- (280.08,177.86) -- (280.36,145.63) -- (303.34,123.04) -- (335.57,123.33) -- cycle ;
\draw  [dashed, thin] (456.88,178.54) -- (433.89,201.12) -- (379.08,177.86) -- (379.36,145.63) -- (434.57,123.33) -- (457.16,146.31) -- cycle ;
\draw    (38.08,176.86) -- (93.57,122.33) ;
\draw    (93.57,122.33) -- (38.36,144.63) ;
\draw    (38.36,144.63) -- (38.08,176.86) ;
\draw    (38.08,176.86) -- (115.88,177.54) ;
\draw    (115.88,177.54) -- (116.16,145.31) ;
\draw    (93.57,122.33) -- (116.16,145.31) ;
\draw    (61.34,122.04) -- (93.57,122.33) ;
\draw    (38.36,144.63) -- (61.34,122.04) ;
\draw    (61.34,122.04) -- (115.88,177.54) ;
\draw    (61.34,122.04) -- (116.16,145.31) ;
\draw    (61.34,122.04) -- (38.08,176.86) ;
\draw    (38.08,176.86) -- (116.16,145.31) ;
\draw    (38.36,144.63) -- (115.88,177.54) ;
\draw    (93.57,122.33) -- (115.88,177.54) ;
\draw    (38.36,144.63) -- (116.16,145.31) ;
\draw    (160.08,177.86) -- (215.57,123.33) ;
\draw    (215.57,123.33) -- (160.36,145.63) ;
\draw    (160.36,145.63) -- (160.08,177.86) ;
\draw    (160.08,177.86) -- (237.88,178.54) ;
\draw    (237.88,178.54) -- (238.16,146.31) ;
\draw    (215.57,123.33) -- (238.16,146.31) ;
\draw    (183.34,123.04) -- (215.57,123.33) ;
\draw    (160.36,145.63) -- (183.34,123.04) ;
\draw    (183.34,123.04) -- (237.88,178.54) ;
\draw    (183.34,123.04) -- (238.16,146.31) ;
\draw    (183.34,123.04) -- (160.08,177.86) ;
\draw    (160.08,177.86) -- (238.16,146.31) ;
\draw    (160.36,145.63) -- (237.88,178.54) ;
\draw    (215.57,123.33) -- (237.88,178.54) ;
\draw    (160.36,145.63) -- (238.16,146.31) ;
\draw    (280.36,145.63) -- (280.08,177.86) ;
\draw    (280.36,145.63) -- (335.57,123.33) ;
\draw    (280.08,177.86) -- (335.57,123.33) ;
\draw    (379.36,145.63) -- (379.08,177.86) ;
\draw    (379.36,145.63) -- (434.57,123.33) ;
\draw    (379.08,177.86) -- (434.57,123.33) ;
\draw    (457.16,146.31) -- (456.88,178.54) ;
\draw    (214.89,201.12) -- (237.88,178.54) ;
\draw    (160.08,177.86) -- (214.89,201.12) ;
\draw    (160.36,145.63) -- (214.89,201.12) ;
\draw    (183.34,123.04) -- (214.89,201.12) ;
\draw    (215.57,123.33) -- (214.89,201.12) ;
\draw    (214.89,201.12) -- (238.16,146.31) ;

\draw (41,211.4) node [anchor=north west][inner sep=0.75pt]    {$G_{\mathcal{D} \cup \left( C_{1} \cdot D_{1}^{1}\right)}$};
\draw (162,212.4) node [anchor=north west][inner sep=0.75pt]    {$G_{\mathcal{D} \cup \left( C_{1} \cdot D_{2}^{1}\right)}$};
\draw (281,212.4) node [anchor=north west][inner sep=0.75pt]    {$G_{\mathcal{D} +F_{1}^{1}}$};
\draw (380,212.4) node [anchor=north west][inner sep=0.75pt]    {$G_{\mathcal{D} +F_{2}^{1}}$};
\draw (118.02,164.82) node [anchor=north west][inner sep=0.75pt]    {$1$};
\draw (240.02,165.82) node [anchor=north west][inner sep=0.75pt]    {$1$};
\draw (459.02,165.82) node [anchor=north west][inner sep=0.75pt]    {$1$};
\draw (117,138.4) node [anchor=north west][inner sep=0.75pt]    {$2$};
\draw (239,138.4) node [anchor=north west][inner sep=0.75pt]    {$2$};
\draw (459,138.4) node [anchor=north west][inner sep=0.75pt]    {$2$};
\draw (91,107) node [anchor=north west][inner sep=0.75pt]    {$3$};
\draw (211,107) node [anchor=north west][inner sep=0.75pt]    {$3$};
\draw (331,107) node [anchor=north west][inner sep=0.75pt]    {$3$};
\draw (431,107) node [anchor=north west][inner sep=0.75pt]    {$3$};
\draw (57,107) node [anchor=north west][inner sep=0.75pt]    {$4$};
\draw (177,107) node [anchor=north west][inner sep=0.75pt]    {$4$};
\draw (298,107) node [anchor=north west][inner sep=0.75pt]    {$4$};
\draw (25,137.62) node [anchor=north west][inner sep=0.75pt]    {$5$};
\draw (147,138.62) node [anchor=north west][inner sep=0.75pt]    {$5$};
\draw (267,138.62) node [anchor=north west][inner sep=0.75pt]    {$5$};
\draw (366,138.62) node [anchor=north west][inner sep=0.75pt]    {$5$};
\draw (25,163.62) node [anchor=north west][inner sep=0.75pt]    {$6$};
\draw (147,163.62) node [anchor=north west][inner sep=0.75pt]    {$6$};
\draw (267,163.62) node [anchor=north west][inner sep=0.75pt]    {$6$};
\draw (366,162.62) node [anchor=north west][inner sep=0.75pt]    {$6$};
\draw (94.89,203.52) node [anchor=north west][inner sep=0.75pt]    {$8$};
\draw (216.89,204.52) node [anchor=north west][inner sep=0.75pt]    {$8$};
\draw (336.89,204.52) node [anchor=north west][inner sep=0.75pt]    {$8$};
\draw (435.89,204.52) node [anchor=north west][inner sep=0.75pt]    {$8$};
\end{tikzpicture}
\end{center}
After this step, we have $\mathcal{G} = \{\mD, \mD \cup (C_1 \cdot D^1_1), \mD \cup (C_1 \cdot D^1_2), \mD +F^1_1, \mD + F^1_2\}$. We now do the same for the connected component $C_2$ and obtain:
\[ P_{8,\{7\}} \backslash C_2 = D^2_1 \cup D^2_2, \quad C_2|_{P_{8,\{7\}}} = F^2_1\cup F^2_2 \]
with $V(D^2_1)=\{4\}$, $V(D^2_2) = \{1,2,8\}$ and $V(F^2_1)=\{3\}$, $V(F^2_2)=\{5,6\}$.

\medskip
Repeating the same procedure, after the first cycle of the algorithm we get:
\[ \mathcal{G} = \left\{\mD, \mD \cup (C_1 \cdot D_1^1), \binom{[n]}{2}, \mD + F^1_1, \mD + F^1_2, \mD + F^2_1, \mD +F^2_2 \right\}. \]
We remove $\mD$ from the set $\mathcal{G}$. We then proceed to repeat the same procedure for every element in $\mathcal{G}$. We now determine the nice sets in $\mathcal{G}$ and add them to the positroid set $\Pos(\mD)$:
\[ \Pos(\mD) = \left\{ (\mD \cup (C_1 \cdot D_1^1))^c, \emptyset, (\mD + F^1_2)^c, (\mD + F^2_1)^c \right\}.\]
The only sets in $\mathcal{G}$ that are not nice are $\mathcal{F} = \mD+F^1_1$, $\mathcal{F}' = \mD + F^2_2$. Denoting with a tilde the connected components in $G_\mF$ and $G_\mF'$ corresponding to the ones in $G_\mD$ and repeating the same procedure for these two sets we obtain that the algorithm returns the set:
\begin{align*}
    \Pos(\mD) = \{& (\mD \cup (C_1 \cdot D_1^1))^c, \emptyset, (\mD + F^1_2)^c, (\mD + F^2_1)^c, (\mathcal{F} \cup (\tilde{C}_2 \cdot \tilde{D}^2_1))^c , (\mathcal{F} \cup (\tilde{C}_2  \cdot  \tilde{D}_2^2))^c, \\
    & (\mathcal{F}+ \tilde{F}^2_1)^c, (\mathcal{F}+\tilde{F}^2_2)^c, (\mathcal{F}' \cup (\tilde{C}_1 \cdot \tilde{D}^1_1))^c , (\mathcal{F}' \cup (\tilde{C}_1  \cdot  \tilde{D}_1^2))^c, (\mathcal{F}'+ \tilde{F}^1_2)^c \}. 
\end{align*}

Note that there might be non-maximal positroids in $\mD^c$. For instance, the empty set is clearly contained in every other set in $\Pos(\mD)$. Moreover, $\mD\cup(C_2 \cdot D^2_1) \subsetneq \mathcal{F}\cup (\tilde{C}_2 \cdot \tilde{D}^2_1)$ and $
\mD\cup(C_2 \cdot D^2_1) \subsetneq \mF'\cup(\tilde{C}_1 \cdot \tilde{D}_1^1)$. Hence, we remove the non-maximal positroids from $\Pos(\mD)$ to obtain:
\begin{align*}
   \MPos(\mD)= \{& (\mD \cup (C_1 \cdot D_1^1))^c, (\mD + F^1_2)^c, (\mD + F^2_1)^c, (\mathcal{F} \cup (\tilde{C}_2  \cdot  \tilde{D}_2^2))^c, \\
    &(\mathcal{F}+ \tilde{F}^2_1)^c, (\mathcal{F}+\tilde{F}^2_2)^c, (\mathcal{F}' \cup (\tilde{C}_1  \cdot  \tilde{D}_2^1))^c, (\mathcal{F}'+ \tilde{F}^1_2)^c \}. \tag*{\demo}
\end{align*}
\end{example}

As an immediate corollary of Proposition~\ref{prop:matT}, 
we obtain all the {\em maximal positroids} contained in any~set~$\mD$, specifically, for $\mD$ that do not already define matroids.

\begin{corollary}\label{cor:allpos}
Let $\mD \subset \binom{[n]}{2}$ when $\mD^c$ is not necessarily a matroid and $\mathbb{T}_\mD=\{T \mid c(\mD+T) > c(\mD + (T\backslash \{i\}) \text{ for all } i \in T\}$. Then 
$\MPos(\mD) \subseteq \bigcup_{T\in \mathbb{T}_\mD} \Pos\left(\overline{\mD+T}\right).$
\end{corollary}

\begin{remark}
In Algorithm~\ref{alg:posT}, we have significantly restricted the number of operations, identified in Remark~\ref{rmk:twooperations}, leading to the calculation of all the maximal positroids contained in $\mD^c$. More precisely, applying the graph operations in Remark~\ref{rmk:twooperations} lead to many more positroids comparing to the output of Algorithm~\ref{alg:posT}, however, Proposition~\ref{prop:posT} shows that the output contains all the maximal positroids.
Moreover, we can easily detect the non-maximal positroids in $\Pos(\mD)$ by either checking the dependent sets by inclusion or applying the characterization given in Lemma~\ref{lemma:inclusion}.
\end{remark}

\section{Properties of positroid cells}
We now apply the results of \S\ref{sec:results} to compute the dimension, boundaries and intersections of~positroid~cells.
\vspace{-5mm}
\subsection{Dimensions of positroid cells}\label{sec:dim}
\begin{proposition} \label{cor:dimension}
The dimension of the positroid cell $S_+(\mD^c)$ is $n-|T_\mD|+c(\mD)-4$.
\end{proposition}
\begin{proof}
Consider the case $T_\mD=\emptyset$. Following the notation in the proof of Proposition~\ref{prop:niceness}, we have that the associated Le diagram has the following form:
\medskip
\begin{center}
    \begin{tikzpicture}[scale=0.6]
    \draw (0,0) --(0,2) --(8,2) --(8,1) --(4,1) --(4,0) --cycle;
    \draw (0,1) --(4,1) --(4,2);
    \draw (1,0) --(1,2);
    \draw (3,0) --(3,2);
    \draw (5,1) --(5,2);
    \draw (7,1) --(7,2);
    \draw (0.5,0) node[below]{$n$};
    \draw (3.5,0) node[below]{$k+2$};
    \draw (4,0.5) node[right]{$k+1$};
    \draw (7.5,1) node[below]{$2$};
    \draw (8,1.5) node[right]{$1$};
    \draw (0.5,0.5) node{$+$};
    \draw (3.5,0.5) node{$+$};
    \draw (4.5,1.5) node{$+$};
    \draw (7.5,1.5) node{$+$};
    \draw (2,0.5) node{$\dots$};
    \draw (6,1.5) node{$\dots$};
    \end{tikzpicture}
\end{center}
There are $n-2$ $``+"$ in the diagram. We also add a $``+"$ in the box $(1,j+1)$ for every $j \in \{k+1,\dots,n-1\}$ where $\{j,j+1\}\in \mD^c$. Note that there are as many subsets $\{j,j+1\}$ in $\mD^c$ as the number of connected components of the graph $G_\mD$, which we denote by $c(\mD)$. Moreover among these pairs, there are the subsets $\{1,n\}$ and $\{k,k+1\}$ which correspond to the boxes of Le diagram containing a $``+"$ that have already been counted among the previous $n-2$. Since the dimension of the positroid cell equals to the number of $``+"$s in the Le diagram, the dimension is $n-2+c(\mD)-2=n+c(\mD)-4$.

Now assume that $T\neq \emptyset$. As shown in the proof of Proposition~\ref{prop:niceness}, we can see  $\mD^c$ as a positroid on the ground set $[n]\backslash T$. For such a positroid we will have $T_\mD=\emptyset$, and so its dimension is given by $(n-|T|)+c-4$, where $c$ is the number of connected components of $G_\mD$ on the set $[n]\backslash T$, that is $c(\mD)$.
\end{proof}

\begin{example}\label{ex:dim}
The graphs associated to the maximal positroids in Example~\ref{ex:algorithm} are:
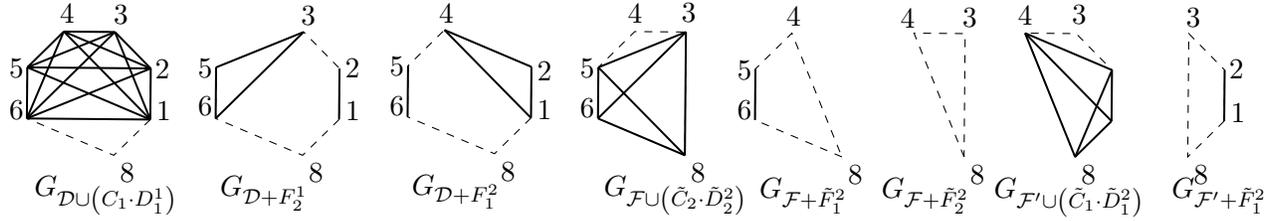
\begin{center}
\tikzset{every picture/.style={line width=0.75pt}} 
\begin{tikzpicture}[x=0.6pt,y=0.6pt,yscale=-1,xscale=1]
\draw  [dashed, thin] (97.88,193.54) -- (74.89,216.12) -- (20.08,192.86) -- (20.36,160.63) -- (43.34,138.04) -- (75.57,138.33) -- (98.16,161.31) -- cycle ;
\draw    (20.08,192.86) -- (75.57,138.33) ;
\draw    (75.57,138.33) -- (20.36,160.63) ;
\draw    (20.36,160.63) -- (20.08,192.86) ;
\draw    (20.08,192.86) -- (97.88,193.54) ;
\draw    (97.88,193.54) -- (98.16,161.31) ;
\draw    (75.57,138.33) -- (98.16,161.31) ;
\draw    (43.34,138.04) -- (75.57,138.33) ;
\draw    (20.36,160.63) -- (43.34,138.04) ;
\draw    (43.34,138.04) -- (97.88,193.54) ;
\draw    (43.34,138.04) -- (98.16,161.31) ;
\draw    (43.34,138.04) -- (20.08,192.86) ;
\draw    (20.08,192.86) -- (98.16,161.31) ;
\draw    (20.36,160.63) -- (97.88,193.54) ;
\draw    (75.57,138.33) -- (97.88,193.54) ;
\draw    (20.36,160.63) -- (98.16,161.31) ;
\draw  [dashed, thin] (216.88,193.54) -- (193.89,216.12) -- (139.08,192.86) -- (139.36,160.63) -- (194.57,138.33) -- (217.16,161.31) -- cycle ;
\draw    (139.36,160.63) -- (139.08,192.86) ;
\draw    (139.36,160.63) -- (194.57,138.33) ;
\draw    (139.08,192.86) -- (194.57,138.33) ;
\draw    (217.16,161.31) -- (216.88,193.54) ;
\draw  [dashed, thin] (337.88,192.54) -- (314.89,215.12) -- (260.08,191.86) -- (260.36,159.63) -- (283.34,137.04) -- (338.16,160.31) -- cycle ;
\draw    (260.36,159.63) -- (260.08,191.86) ;
\draw    (283.34,137.04) -- (337.88,192.54) ;
\draw    (283.34,137.04) -- (338.16,160.31) ;
\draw    (338.16,160.31) -- (337.88,192.54) ;
\draw  [dashed, thin] (434.89,216.12)-- (380.08,192.86) -- (380.36,160.63) -- (403.34,138.04) -- (435.57,138.33) -- cycle ;
\draw    (380.36,160.63) -- (380.08,192.86) ;
\draw    (435.57,138.33) -- (434.89,216.12) ;
\draw    (380.36,160.63) -- (435.57,138.33) ;
\draw    (380.08,192.86) -- (435.57,138.33) ;
\draw    (380.36,160.63) -- (434.89,216.12) ;
\draw    (380.08,192.86) -- (434.89,216.12) ;
\draw  [dashed, thin] (533.89,217.12) -- (479.08,193.86) -- (479.36,161.63) -- (502.34,139.04)-- cycle ;
\draw    (479.36,161.63) -- (479.08,193.86) ;
\draw  [dashed, thin] (610.89,217.12) -- (579.34,139.04) -- (611.57,139.33) -- cycle ;
\draw  [dashed, thin] (703.88,194.54) -- (680.89,217.12) -- (649.34,139.04) -- (681.57,139.33) -- (704.16,162.31) -- cycle ;
\draw    (704.16,162.31) -- (680.89,217.12) ;
\draw    (649.34,139.04) -- (703.88,194.54) ;
\draw    (704.16,162.31) -- (649.34,139.04) ;
\draw    (649.34,139.04) -- (680.89,217.12) ;
\draw    (703.88,194.54) -- (680.89,217.12) ;
\draw    (703.88,194.54) -- (704.16,162.31) ;
\draw  [dashed, thin] (774.88,194.54) -- (751.89,217.12) -- (752.57,139.33) -- (775.16,162.31) -- cycle ;
\draw    (775.16,162.31) -- (774.88,194.54) ;

\draw (23,227.4) node [anchor=north west][inner sep=0.75pt]    {$G_{\mathcal{D} \cup \left( C_{1} \cdot D_{1}^{1}\right)}$};
\draw (100.02,180.82) node [anchor=north west][inner sep=0.75pt]    {$1$};
\draw (99,154.4) node [anchor=north west][inner sep=0.75pt]    {$2$};
\draw (73,119.4) node [anchor=north west][inner sep=0.75pt]    {$3$};
\draw (39,119.02) node [anchor=north west][inner sep=0.75pt]    {$4$};
\draw (7,153.62) node [anchor=north west][inner sep=0.75pt]    {$5$};
\draw (7,179.62) node [anchor=north west][inner sep=0.75pt]    {$6$};
\draw (76.89,219.52) node [anchor=north west][inner sep=0.75pt]    {$8$};
\draw (140,227.4) node [anchor=north west][inner sep=0.75pt]    {$G_{\mathcal{D} +F_{2}^{1}}$};
\draw (219.02,180.82) node [anchor=north west][inner sep=0.75pt]    {$1$};
\draw (219,153.4) node [anchor=north west][inner sep=0.75pt]    {$2$};
\draw (191,120.4) node [anchor=north west][inner sep=0.75pt]    {$3$};
\draw (126,153.62) node [anchor=north west][inner sep=0.75pt]    {$5$};
\draw (126,177.62) node [anchor=north west][inner sep=0.75pt]    {$6$};
\draw (195.89,219.52) node [anchor=north west][inner sep=0.75pt]    {$8$};
\draw (261,226.4) node [anchor=north west][inner sep=0.75pt]    {$G_{\mathcal{D} +F_{1}^{2}}$};
\draw (278,119.02) node [anchor=north west][inner sep=0.75pt]    {$4$};
\draw (247,152.62) node [anchor=north west][inner sep=0.75pt]    {$5$};
\draw (247,177.62) node [anchor=north west][inner sep=0.75pt]    {$6$};
\draw (316.89,218.52) node [anchor=north west][inner sep=0.75pt]    {$8$};
\draw (340.02,179.82) node [anchor=north west][inner sep=0.75pt]    {$1$};
\draw (340,155.4) node [anchor=north west][inner sep=0.75pt]    {$2$};
\draw (382,227.4) node [anchor=north west][inner sep=0.75pt]    {$G_{\mathcal{F} \cup \left(\tilde{C}_{2} \cdot\tilde{D}_{2}^{2}\right)}$};
\draw (431,118.4) node [anchor=north west][inner sep=0.75pt]    {$3$};
\draw (397,119.02) node [anchor=north west][inner sep=0.75pt]    {$4$};
\draw (367,153.62) node [anchor=north west][inner sep=0.75pt]    {$5$};
\draw (367,178.62) node [anchor=north west][inner sep=0.75pt]    {$6$};
\draw (436.89,219.52) node [anchor=north west][inner sep=0.75pt]    {$8$};
\draw (480,228.4) node [anchor=north west][inner sep=0.75pt]    {$G_{\mathcal{F} +\tilde{F}_{1}^{2}}$};
\draw (497,121.02) node [anchor=north west][inner sep=0.75pt]    {$4$};
\draw (466,154.62) node [anchor=north west][inner sep=0.75pt]    {$5$};
\draw (466,179.62) node [anchor=north west][inner sep=0.75pt]    {$6$};
\draw (535.89,220.52) node [anchor=north west][inner sep=0.75pt]    {$8$};
\draw (557,228.4) node [anchor=north west][inner sep=0.75pt]    {$G_{\mathcal{F} +\tilde{F}_{2}^{2}}$};
\draw (608,121.4) node [anchor=north west][inner sep=0.75pt]    {$3$};
\draw (612.89,220.52) node [anchor=north west][inner sep=0.75pt]    {$8$};
\draw (569,122.02) node [anchor=north west][inner sep=0.75pt]    {$4$};
\draw (628,228.4) node [anchor=north west][inner sep=0.75pt]    {$G_{\mathcal{F'} \cup \left(\tilde{C}_{1} \cdot\tilde{D}_{1}^{2}\right)}$};
\draw (677,119.4) node [anchor=north west][inner sep=0.75pt]    {$3$};
\draw (643,120.02) node [anchor=north west][inner sep=0.75pt]    {$4$};
\draw (682.89,220.52) node [anchor=north west][inner sep=0.75pt]    {$8$};
\draw (740,228.4) node [anchor=north west][inner sep=0.75pt]    {$G_{\mathcal{F'} +\tilde{F}_{1}^{2}}$};
\draw (749,121.4) node [anchor=north west][inner sep=0.75pt]    {$3$};
\draw (753.89,220.52) node [anchor=north west][inner sep=0.75pt]    {$8$};
\draw (777.27,181.82) node [anchor=north west][inner sep=0.75pt]    {$1$};
\draw (776.23,154.4) node [anchor=north west][inner sep=0.75pt]    {$2$};

\end{tikzpicture}
\end{center}
By applying Proposition~\ref{cor:dimension}, we can compute the dimensions of these positroids as follows:
\begin{align*}
    & \dim \left( (\mD \cup (C_1 \cdot D_1^1))^c \right) = 8-1+2-4 = 5\\
    & \dim \left( (\mD + F^1_2)^c \right) = \dim \left( (\mD + F^2_1)^c \right) = 8-2+3-4 = 5\\
    & \dim \left( (\mathcal{F} \cup (\tilde{C}_2  \cdot  \tilde{D}_2^2))^c \right) = \dim \left( (\mathcal{F}' \cup (\tilde{C}_1  \cdot  \tilde{D}_1^2))^c \right) = 8-3+2-4 = 3\\
    & \dim \left( (\mathcal{F}+ \tilde{F}^2_1)^c \right) = \dim \left( (\mathcal{F}'+ \tilde{F}^1_2)^c \right) = 8-4+3-4 = 3\\
    & \dim \left( (\mathcal{F}+\tilde{F}^2_2)^c \right) = 8-5+3-4 = 2. 
\end{align*}
Note that the maximal positroids contained in a given set may have different dimensions. \demo
\end{example}

\begin{remark}
It is interesting to notice that the elements of $\MPos(\mD)$
may have different dimensions, as seen in Example~\ref{ex:dim}. If we are interested only in the ones with the maximal dimension, then we can easily compute all the dimensions of the sets in $\Pos(\mD)$ using Proposition~\ref{cor:dimension} and select those with maximal dimensions. This way, we obtain all and only the positroids of maximal dimension contained in $\mD^c$.
\end{remark}

\subsection{Boundaries of positroid cells \label{sec:boundaries}} 
We now apply our results to compute all the codimension one boundaries of a given positroid cell which are themselves are positroids. Similarly, by applying Proposition~\ref{res:1boundaries} multiple times, one can compute all codimension $k$ boundaries. To date, given any of the cryptomorphic representations of positroid cells (e.g.~Grassmann necklaces, Le diagrams, plabic graphs, etc) there is no explicit algorithm to determine all the boundaries of a given codimension for a positroid cell. While one can count the number of ``+'' decorations in a Le diagram to get the dimension cell, one cannot get all the codimension one boundaries of a Le diagram by simply replacing the ``+''s with ``0''s such that the resulting decoration satisfies the Le condition. An example of this is given in \cite[Example 2.17]{agarwala2018study} and mentioned as Example~\ref{ex:dim_count} here. We give a smaller example to illustrate the difficulties and nuances of identifying the boundary cells of a positroid from the various combinatorial representations below.

\begin{example}\label{eg:badLebehavior} 
Consider the matroids with basis sets $\mathcal{B}_1 = \{12,14,24\}$ and $\mathcal{B}_2 = \{12,13,14,23,24\}$. By inclusion of basis sets, we have $\mathcal{B}_1 \subseteq \mathcal{B}_2$. The corresponding Le diagrams are:
\[L(\mathcal{B}_1) = \ytableaushort{+0,+0}\quad \text{ and }\quad L(\mathcal{B}_2) = \ytableaushort{0+,++}\ .\] 
The positroid cell corresponding to $\mathcal{B}_2$ has dimension three, while the positroid cell corresponding to $\mathcal{B}_1$ has dimension two. Inclusion of basis sets implies that the positroid of $\mathcal{B}_1$ is a codimension-one boundary of the positroid defined by $\mathcal{B}_2$. However, there is no way to obtain this boundary by simply taking a subset of ``$+$’’s in $L(\mathcal{B}_2)$ that respects the Le condition.

In this case, one can see the boundary structure by considering subexpression–expression pairs. Recall from Section~\ref{sec:pre} that for $\mathbf{v} = s_2 s_3 s_1 s_2$, a reduced expression in $S_n$, the positive subexpression $\mathbf{u} = \varepsilon \varepsilon \varepsilon s_2$ corresponds to $\mathcal{B}_2$, and $\mathbf{u}' = s_2 \varepsilon s_1 \varepsilon$ corresponds to $\mathcal{B}_1$. Then, since $\mathbf{v}$ is fixed and $\mathbf{u}'$ is a reduced subexpression of $\mathbf{u}$, the condition $\mathbf{u} \leq \mathbf{u}' \leq \mathbf{v}$ holds, and we see that the positroid $(\mathbf{u}', \mathbf{v})$ is a boundary of $(\mathbf{u}, \mathbf{v})$; see e.g.~\cite{marcott2020combinatorics, talaska2013network, knutson2013positroid}. 

However, note that the above approach for identifying boundaries requires a fixed $\mathbf{v}$, i.e., a fixed Le shape. It cannot be applied when the shape changes. While this approach correctly identifies the boundary above, it fails to identify boundaries where the Le shape changes. For instance, it misses that the positroid \[ \ytableaushort{0+,+}  \text{ with basis set } \{12, 13, 34\} \text{ and subexpression–expression pair } (\varepsilon \varepsilon s_2 , s_3 s_1 s_2) \] is a boundary of the positroid \[ \ytableaushort{++,0+}  \text{ with basis set } \{12, 13, 24, 34\} \text{ and subexpression–expression pair } (\varepsilon s_3 \varepsilon \varepsilon , s_2 s_3 s_1 s_2),\] even though this relationship is clear from the basis sets.
\end{example}

\begin{remark}
    While unpublished work by Tower \cite{Towercommunication} provides a method for calculating all such boundaries from Bruhat intervals, there is no direct way to calculate the dimension of a given subexpression–expression pair without passing through the corresponding Le diagram. Moreover, although algorithms exist for passing from a positroid, viewed from a matroidal perspective (i.e., in terms of basis sets), to Le diagrams, and for passing from subexpression–expression pairs to Le diagrams, it remains very difficult to determine the corresponding pairs directly from dependency information in order to calculate these boundaries. Proposition~\ref{res:1boundaries} fills this gap in the literature for rank-$2$ positroids.
\end{remark}

The main idea of the proposition below is to use the dimension formula 
in~Proposition~\ref{cor:dimension}~to determine the types of changes that can be made to a nice set $\mD$ to form a larger nice set $\mD'$ of the correct~dimension.

\begin{proposition}\label{res:1boundaries}
Let $\mD\subset \binom{[n]}{2}$ be such that $\mD^c$ is a positroid with associated graph $G_\mD$. Let $C_1, \dots, C_m$ be the connected components of $G_\mD$ which are cyclically ordered.\footnote{i.e.~$C_{m+1}=C_1$ and for every $i \in [m]$ there is some $j \in [n]$ such that $j \in V(C_i)$ and $j+1 \in V(C_{i+1})$.} Consider the subset $J=\{j \in [n]\backslash T_\mD \mid j \in V(C_i) \text{ with } |V(C_i)|>1 \text{ for some } i \in [m]\}$. Then the set
\[
    \partial\Pos_1(\mD) =  \{ \mD\cup (C_i \cdot C_{i+1}) \mid i \in [m]\} \cup \bigcup_{j \in J} \{\mD+\{j\}\}
\] 
is the set of all positroids $\mathcal{P}$ such that $\mathcal{P}\subset \mD^c$ and $\dim \mathcal{P} = \dim \left( \mD^c\right) -1$.
\end{proposition}
\begin{proof}
By Proposition~\ref{cor:dimension}, the dimension of the positroid cell corresponding to $\mD$ is $n-|T_\mD|+c(\mD)-4$. 
Assume that $\mD'\subset \binom{[n]}{2}$ is such that $(\mD')^c$ is a positroid contained in $\mD^c$ and $\dim ((\mD')^c) = \dim(\mD^c)-1$. 
We know that $\mD \subset \mD'$ and that $\mD'$ is nice. Hence $\mD'$  satisfies $n-|T_{\mD'}|-4+c(\mD') = \dim (\mD^c)-1.$ In other words, there is some 
integer $b < c(\mD)$ such that  $c(\mD') = c(\mD) -b$ and $|T_{\mD'}| = |T_{\mD}| - b+ 1$. Note that $c(\mD') \leq c(\mD)$, since all the connected components are complete, and so we can not separate one connected component by adding a vertex to the vanishing set $T_{\mD'}$. Hence $b\geq 0$. Moreover, $T_\mD \subseteq T_{\mD'}$ implies that $|T_{\mD}| \leq |T_{\mD'}| = |T_{\mD}| - b+ 1$, and so $b$ can be either $0$ or $1$.

For $b=1$, the number of connected components decreases by $1$. 
The only possibilities to obtain a nice set is to merge two consecutive connected components, $C_i$ and $C_{i+1}$ for some $i$. On the other hand, for $b=0$, we are adding 
an element to the vanishing set $T_\mD$ without changing the number of connected components of $G_\mD$, so the only possibilities are the vertices lying in components with at least $2$ vertices, that is $\{i \in [n]\backslash T_\mD \mid \text{ there exists } j \in [n]\backslash T_\mD \text{ such that } \{i,j\} \in \mD\}$. These are exactly the graphs corresponding to the codimension one positroids contained in the boundary of $\mD^c$. 
\end{proof}

\subsection{Intersections of positroid cells \label{sec:intersections}}
We now show that the maximal positroids in the intersection, or the common boundary of two or more positroid cells are exactly the maximal matroids contained in the intersection. That is, applying Proposition~\ref{prop:matT} to the union of the dependent sets produces all the maximal positroids. 

\begin{proposition} \label{prop:intersection}
Let $\mD_1, \dots, \mD_k \subset \binom{[n]}{2}$ be the dependent sets of $k$ positroids. The set of the maximal positroids lying in the common boundary of the positroid cells associated to $(\mD_1)^c, \dots, (\mD_k)^c$ is equal to:
\[ \MPos(\mD_1\cup\dots\cup\mD_k) = \Mat(\mD_1\cup \dots\cup\mD_k) = \left\{ \left( \overline{(\mD_1\cup \dots \cup \mD_k) +T} \right)^c \mid T \in \mathbb{T}_{\mD_1\cup\dots\cup\mD_k} \right\}  \]
\end{proposition}
\begin{proof}
Consider the set $\mD = \mD_1\cup \dots \cup \mD_k$. Let $T \in \mathbb{T}_{\mD}$ and $\mathcal{F} = \overline{\mD + T}$. We want to prove that $\mF$ is a positroid. Since by construction of $\mathcal{F}$ all the connected components of $G_\mF$ are complete, we only need to prove that for every connected component $C$ of $G_\mF$, $P_{n,T_\mF}\backslash C$ is connected. Suppose by contradiction this is not true and let $C$ be a connected component of $G_\mF$ such that $P_{n,T_\mF}\backslash C$ is disconnected. Let $[a_1,b_1], \dots, [a_m,b_m]$ be the maximal cyclic intervals in $C|_{P_{n,T_\mF}}$. Note that since $C$ is a complete component, for every $e \in [m]$ there exists $f_e \in [n]\backslash T_\mF$ such that $b_e <  f_e < a_{e+1}$ (with respect to cyclic order). Since $C$ is a connected component, there must exist $c \in [a_i,b_i]$ and $d\in [a_j,b_j]$ for some $i<j$, $i,j \in [m]$, such that $\{c,d\} \in \mD = \mD_1\cup \dots \cup \mD_k$. In particular, $\{c,d\} \in \mD_\ell$ for some $\ell \in [k]$. Since $(\mD_\ell)^c$ is a positroid, if $C'$ is the connected component of $G_{\mD_\ell}$ containing $\{c,d\}$, then one of the cyclic interval of $[n]\backslash T_{\mD_\ell}$, $[c,d]$ or $[d,c]$ is contained in $V(C')$. Assume that $[c,d]\subset V(C')$. Then, since $C'$ is complete, $[c,d]|_{[n]\backslash T_\mF} \subset \bigcup_{i\leq e \leq j}[a_e,b_e]$. This is not possible since $f_i \in [c,d]|_{[n]\backslash T_\mF} \backslash \bigcup_{i\leq e \leq j}[a_e,b_e]$. In the same way, if $[d,c]\subset V(C')$ we would get a contradiction.
It follows that $\mF = \overline{\mD+T}$ is the dependent set of a positroid and hence, applying Algorithm~\ref{alg:posT} to $\mF$ we would get $\Pos(\mF) = \{\mF^c\}$. In particular we have that 
\begin{equation*}
\MPos(\mD) = \bigcup_{T \in \mathbb{T}_\mD} \Pos\left(\overline{\mD+T}\right) = \bigcup_{T \in \mathbb{T}_\mD} \left\{\left(\overline{\mD+T}\right)^c\right\} = \Mat(\mD). \tag*{\qedhere}\end{equation*}
\end{proof}

We now consider Example~4.4 from \cite{agarwala2018study}, where the authors presented a pair of positroids whose common boundary contains a codimension-one positroid that cannot be obtained from their corresponding Le diagrams. In the example below, we show that the new graphical characterization of positroids simplifies the visualization of such cells.

\begin{example}\label{ex:dim_count}
Consider the subsets $\mD_1=\{\{3,4\}, \{5,6\}\}$ and $\mD_2=\{\{2,3\}, \{5,6\}\}$. By Proposition~\ref{prop:niceness}, we see that $(\mD_1)^c$ and $(\mD_2)^c$ are both positroids. The graph associated to these positroids are:
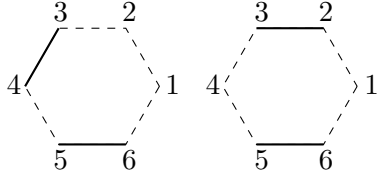
\begin{center}
\tikzset{every picture/.style={line width=0.75pt}}
\begin{tikzpicture}[x=0.4pt,y=0.4pt,yscale=-1,xscale=1]

\draw  [dashed, thin] (189,186.5) -- (157.25,241.49) -- (93.75,241.49) -- (62,186.5) -- (93.75,131.51) -- (157.25,131.51) -- cycle ;
\draw[thick]    (62,186.5) -- (93.75,131.51) ;
\draw[thick]    (93.75,241.49) -- (157.25,241.49) ;
\draw  [dashed, thin] (376,186.5) -- (344.25,241.49) -- (280.75,241.49) -- (249,186.5) -- (280.75,131.51) -- (344.25,131.51) -- cycle ;
\draw[thick]    (344.25,241.49) -- (280.75,241.49) ;
\draw[thick]    (280.75,131.51) -- (344.25,131.51) ;

\draw (192,173.4) node [anchor=north west][inner sep=0.75pt]    {$1$};
\draw (380,173.4) node [anchor=north west][inner sep=0.75pt]    {$1$};
\draw (151,105) node [anchor=north west][inner sep=0.75pt]    {$2$};
\draw (337,105) node [anchor=north west][inner sep=0.75pt]    {$2$};
\draw (86,105) node [anchor=north west][inner sep=0.75pt]    {$3$};
\draw (276,105) node [anchor=north west][inner sep=0.75pt]    {$3$};
\draw (42,173.4) node [anchor=north west][inner sep=0.75pt]    {$4$};
\draw (230,173.4) node [anchor=north west][inner sep=0.75pt]    {$4$};
\draw (86,244.4) node [anchor=north west][inner sep=0.75pt]    {$5$};
\draw (276,244.4) node [anchor=north west][inner sep=0.75pt]    {$5$};
\draw (151,244.4) node [anchor=north west][inner sep=0.75pt]    {$6$};
\draw (337,244.4) node [anchor=north west][inner sep=0.75pt]    {$6$};
\end{tikzpicture}
\end{center}
and we have $\dim(\mD_1^c) = \dim(\mD_2^c)= 6+4-4=6$.
In order to compute the positroids contained in their common boundaries, we use Corollary~\ref{cor:allpos} to determine all the maximal positroids whose dependent set contains $\mD=\mD_1\cup\mD_2$. We have $\mathbb{T}_\mD = \{\emptyset, \{3\}\}$ and 
\[\Mat(\mD)=\left\{\left(\overline{\mD_1\cup\mD_2}\right)^c, \left(\overline{(\mD_1\cup\mD_2)+\{3\}}\right)^c\right\}.\] 
The dimension of these positroid cells are $6+3-4=5$ and $6-1+4-4=5$. The associated graphs are:
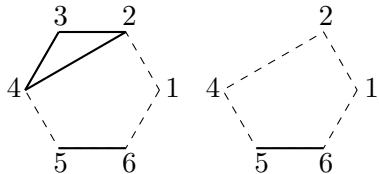
\begin{center}
\tikzset{every picture/.style={line width=0.75pt}} 
\begin{tikzpicture}[x=0.4pt,y=0.4pt,yscale=-1,xscale=1]
\draw  [dashed,thin] (189,186.5) -- (157.25,241.49) -- (93.75,241.49) -- (62,186.5) -- (93.75,131.51) -- (157.25,131.51) -- cycle ;
\draw[thick]    (62,186.5) -- (93.75,131.51) ;
\draw[thick]    (93.75,241.49) -- (157.25,241.49) ;
\draw  [dashed,thin] (376,186.5) -- (344.25,241.49) -- (280.75,241.49) -- (249,186.5)-- (344.25,131.51) -- cycle ;
\draw[thick]    (344.25,241.49) -- (280.75,241.49) ;
\draw[thick]    (93.75,131.51) -- (157.25,131.51) ;
\draw[thick]    (62,186.5) -- (157.25,131.51) ;

\draw (192,173.4) node [anchor=north west][inner sep=0.75pt]    {$1$};
\draw (380,173.4) node [anchor=north west][inner sep=0.75pt]    {$1$};
\draw (151,105) node [anchor=north west][inner sep=0.75pt]    {$2$};
\draw (337,105) node [anchor=north west][inner sep=0.75pt]    {$2$};
\draw (86,105) node [anchor=north west][inner sep=0.75pt]    {$3$};
\draw (42,173.4) node [anchor=north west][inner sep=0.75pt]    {$4$};
\draw (230,173.4) node [anchor=north west][inner sep=0.75pt]    {$4$};
\draw (86,244.4) node [anchor=north west][inner sep=0.75pt]    {$5$};
\draw (276,244.4) node [anchor=north west][inner sep=0.75pt]    {$5$};
\draw (151,244.4) node [anchor=north west][inner sep=0.75pt]    {$6$};
\draw (337,244.4) node [anchor=north west][inner sep=0.75pt]    {$6$};
\end{tikzpicture}
\end{center}
We can see that these are nice sets, hence they are the maximal positroids contained in $(\mD_1)^c \cap (\mD_2)^c$. Moreover, they are both positroids of codimension $1$ lying in the boundary of both $\mD_1^c$ and $\mD_2^c$. \demo
\end{example}

\bibliographystyle{abbrv}
\bibliography{JKresidue}

\end{document}